\documentclass[11pt,oneside]{amsart}
\usepackage[english]{babel}
\usepackage{fullpage}
\usepackage{amssymb,pstricks,amscd,epsfig}
\usepackage{graphicx}

\usepackage{amsmath, amsthm, amssymb}
\usepackage{pdfsync}
\usepackage[latin1]{inputenc}
\usepackage{stmaryrd}
\usepackage{color}
\usepackage{setspace}
\usepackage[matrix,arrow,tips,curve]{xy}

\usepackage{hyperref}

\usepackage{cite}

\newcommand{\Z}{\mathbb{Z}}
\newcommand{\Q}{\mathbb{Q}}

\newcommand{\C}{\mathbb{C}}

\newcommand{\Shhom}{\mc{H}om}

\DeclareMathOperator{\ext}{ext}
\newcommand{\mc}{\mathcal}

\newcommand{\ua}{\underline{\alpha}}

\newcommand{\Mf}{\mathfrak{M}}

\newcommand{\be}{\begin{equation}}
\newcommand{\ee}{\end{equation}}

\newcommand{\wt}{\widetilde}

\newtheorem{thm}{Theorem}

\newtheorem{prop}[thm]{Proposition}
\newtheorem{lemma}[thm]{Lemma}
\newtheorem{defin}[thm]{Definition}
\newtheorem{cor}[thm]{Corollary}

\theoremstyle{definition}

\newtheorem{rem}[thm]{Remark}

\newtheorem{example}[thm]{Example}

\numberwithin{thm}{section}

\numberwithin{equation}{section}

\DeclareMathOperator{\rk}{rk}
\DeclareMathOperator{\Hom}{Hom}
\DeclareMathOperator{\Ext}{Ext}

\DeclareMathOperator{\ch}{ch}
\DeclareMathOperator{\td}{td}

\DeclareMathOperator{\Pic}{Pic}

\DeclareMathOperator{\im}{Im}

\DeclareMathOperator{\Aut}{Aut}

\DeclareMathOperator{\Sing}{Sing}

\DeclareMathOperator{\Def}{Def}

\DeclareMathOperator{\Stab}{Stab}

\newcommand{\famsigma}{{\underline{\sigma}}}

\newcommand{\Ku}{\mc{K}u}

\author[G. Sacc\`a]{Giulia Sacc\`a}
\address{Mathematics Department\\
Columbia University \\
Mathematics Department\\
2990 Broadway\\
New York, NY 10027}

\title{Moduli spaces on Kuznetsov components are Irreducible Symplectic Varieties}

\begin{document}

\maketitle

\begin{abstract} This article studies moduli spaces of Bridgeland semistable objects in the Kuznetsov component of a cubic fourfold that don't admit a symplectic resolution, i.e., moduli spaces of objects with non-primitve Mukai vector $v=mv_0$  that is not of OG10-type and where $v_0^2 >0$. For a generic stability condition, it is shown that these moduli spaces are projective irreducible symplectic varieties with factorial terminal singularities and that their deformation class is uniquely determined by the  integers $m$ and $v_0^2$.
On the one hand, this generalizes the results of \cite{BLMNPS, Perry-Pertusi-Zhao, Li-Pertusi-Zhao}, which deal with moduli spaces of objects in the  Kuznetsov component of a cubic fourfold which are smooth or of OG10-type; on the other hand, this extends to the Kuznetsov component of a cubic fourfold the results of \cite{Perego-Rapagnetta-Irr, Perego-Rapagnetta-second-cohomology} on Gieseker moduli spaces of sheaves on K3 surfaces with non-primitive Mukai vector.

\end{abstract}

\section*{Introduction}

Since their appearance, moduli spaces of sheaves on K3 surfaces have been central objects in algebraic geometry. In \cite{Mukai84}, Mukai showed that moduli spaces of stable sheaves on a polarized K3 surface are smooth and carry a holomorphic symplectic structure. This generalizes the natural symplectic form on the Hilbert scheme of points on a K3 surface \cite{Beauville83, Fujiki}. If the Mukai vector encoding the Chern classes of the sheaves parametrized is primitive and the polarization is generic, then by the work of O'Grady, Yoshioka, and Huybrechts \cite{OGradyWt2,Yoshioka-moduli-abelian,Huybrechts2},  these moduli spaces are smooth  irreducible holomorpholomorphichic symplectic manifolds, deformation equivalent to Hilbert schemes of points on a K3 surface. By \cite{Toda, Bayer-Macri-Proj, Yoshioka-moduli-abelian}, the same holds for moduli spaces of Bridgeland stable objects, if the Mukai vector is primitive and the stability condition is generic. 

Recall that an irreducible holomorphic symplectic manifold is a compact K\"ahler manifold that is simply connected and whose space of holomorphic $2$-forms is spanned by a holomorphic symplectic form. By the Beauville-Bogomolov decomposition theorem, irreducible holomorphic symplectic manifolds are, together with complex tori and strict Calabi-Yau manifolds, the building blocks of compact K\"ahler manifolds with trivial first Chern class. The joint effort of several authors Druel--Greb--Guenancia--Höring--Kebekus--Peternell \cite{Druel-Guenancia-DT, Durel-DT, Greb-Guenancia-Kebekus, Horing-Peternell-foliations}, has resulted in a singular version of the decomposition theorem: it states that projective varieties with klt singularities and numerically trivial canonical class decompose, up to a finite cover that is \'etale in codimension one, into the product of abelian varieties, irreducible symplectic varieties, and irreducible Calabi-Yau manifolds. A normal variety is called symplectic if its smooth locus has a holomorphic symplectic form which extends to a regular holomorphic form on any resolution of singularities.
A  projective symplectic variety $X$  is called an irreducible symplectic variety if for all finite covers $f: Y \to X$ that are \'etale in codimension $1$, the algebra of holomorphic forms on $Y$ is generated by the reflexive pullback of a holomorphic symplectic form on $X$.

When the Mukai vector is not primitive there are two cases. Either $v=2v_0$, with $v_0^2=2$, in which case the moduli space admits a symplectic resolution \cite{OGrady99, Lehn-Sorger}. Such a resolution is deformation equivalent to O'Grady's 10-dimensional exception example \cite{Perego-Rapagnetta-OG10, Meachan-Zhang}.  Otherwise, there is no symplectic resolution:  Kaledin-Lehn-Sorger \cite{Kaledin-Lehn-Sorger} show that if the polarization is generic, the moduli spaces are singular symplectic varieties, with factorial terminal singularities. 
Perego-Rapagnetta prove in \cite{Perego-Rapagnetta-Irr} that these (Gieseker) moduli spaces are irreducible symplectic varieties and that, for fixed $m$ and $v_0^2$, these moduli spaces are all deformation equivalent.
 In \cite{Perego-Rapagnetta-second-cohomology} they further show that Mukai-Donaldson morphism identifies their second integral cohomology group with the orthogonal complement of the Mukai vector. 

Another way of constructing holomorphic symplectic varieties is by considering parameter spaces for curves on a smooth cubic fourfold \cite{Beauville-Donagi, LLSvS, LSV, Voisin-Twisted, IntJac}. Another approach is, moduli theoretic, via the Kuznetsov component. Let $X$ be a smooth cubic fourfold $X$. Its residual category $Ku(X)$ is defined to be the left semi-orthogonal complement of the exceptional collection $\langle \mc O_X,  \mc O_X(1),  \mc O_X(2) \rangle$. It was introduced by Kuznestov  \cite{Kuznetsov-cubic} and it is a so-called K3 category; as the name suggests, it behaves very much like the derived category of a K3 surface. In particular, by a generalized version of Mukai's theorem, moduli spaces of stable objects in $Ku(X)$ admit a holomorphic symplectic form \cite{Kuznetsov-Markushevich}. The existence of Bridgeland stability conditions on these categories was established in \cite{BLMS}; since then, the classical examples of irreducible holomorphic symplectic varieties associated with a smooth cubic fourfold have been interpreted as  moduli spaces of stable objects \cite{Lahoz-Lehn-Macri-Stellari, Li-Pertusi-Zhao-twisted-cubics}.
More generally, in  \cite{BLMNPS}  it is proved that moduli spaces of objects in $Ku(X)$ with primitive Mukai vector and $v$-generic stability condition are projective and they are irreducible holomorphic symplectic manifolds of K3$^{[n]}$-type. The case of moduli spaces of OG10-type is the main object of  \cite{Li-Pertusi-Zhao}. Note that while these OG$10$-type moduli spaces are birational to the twisted intermediate Jacobian fibration \cite{Voisin-Twisted, Li-Pertusi-Zhao}, for very general $X$, the non twisted intermediate Jacobian fibration is not birational to an OG$10$-type moduli space \cite{IntJac}.

The aim of this note is to study the remaining cases, namely, moduli spaces of objects in the Kuznetsov component of a cubic fourfold which do not admit a symplectic resolution.  For generic stability conditions, we show that the Bayer-Macri numerical divisor class on these moduli spaces is ample, so that they are  projective varieties (Theorem \ref{thmproj}) with factorial and terminal singularities (Theorem \ref{locally fact}), and that they are irreducible symplectic varieties (Theorem \ref{thmisv}). We also show that their deformation class is determined by $m$ and $v_0^2$,  and we compute their integral second cohomology group (Theorems \ref{thm defclassk3} and \ref{vperp}). In the last section (\S \ref{sect GM}) we mention the fact that, by using the results of \cite{Perry-Pertusi-Zhao} instead of \cite{BLMNPS}, one can prove the same results for moduli spaces of Bridgelend semistable objects with non-primitive class which is not of OG$10$-type in the Kuznetsov component of Gushel-Mukai varieties of dimension four and six.

In Section \ref{Sect Notation} we briefly recall the notation, basic definitions, and results needed. For sake of brevity instead of giving a complete treatment we give a road map to some of the main results on moduli stacks and moduli spaces of Bridgeland semistable objects that we will use, and refer to other texts for an introduction. 
The ampleness of the Bayer-Macri class is proved in Section \ref{sect ample}, using the tools developed in \cite{BLMNPS}, which we follow closely.
The  factoriality is proved in Section \ref{sect local}, adapting the techniques of \cite{Kaledin-Lehn-Sorger, Drezet, Drezet-Narasimhan} to the case of Bridgeland moduli spaces, which are not global quotients.
In Section \ref{modspacesISV} we prove that these moduli spaces are irreducible symplectic varieties and that the deformation type depends only on the multiplicity and the square of the Mukai vector. This uses, and extends, the corresponding result for Gieseker moduli spaces of Perego-Rapagnetta \cite{Perego-Rapagnetta-Irr}.
Finally in Section \ref{sect second cohomology} we show that the Mukai-Donalson morphism induced by a quasi-universal family induces an isomorphism between the second integral cohomology of the moduli space and the orthogonal complement of the Mukai vector in the Mukai lattice. Again, this extends and uses the corresponding result for Gieseker moduli spaces of Perego-Rapagnetta \cite{Perego-Rapagnetta-second-cohomology}.

\subsection*{Acknowledgements} I am grateful to the participants of the 2019-2020 \emph{Groupe de travail} ``Composantes de Kuznetsov, stabilité, espaces de modules'' for the opportunity of learning together these topics and in particular to E. Macr\`i for answering numerous questions about \cite{BLMNPS}. 
I would also like to thank E. Arbarello, B. Bakker, Ch. Lehn, and J. Alper for useful discussions around the topics of this paper, and the anonymous referee for carefully reading the paper and suggesting improvements to the exposition.

I gratefully acknowledge support from NSF CAREER grant DMS-2144483 and NSF FRG grant DMS-2052750.

\section{Notation, background, and recollection of the main results used} \label{Sect Notation}

In this section we give a quick overview of the main concepts and results that are used in this note, mostly to fix the notation and give references.

For a smooth projective variety $Z$, $D^b(Z)$ will denote the bounded derived category of coherent sheaves on $Z$. We will mostly consider the case when $Z$ is a K3 surface (usually  denoted by $S$), or a smooth cubic fourfold (usually denoted by $X$).
We will also consider the derived category of a K3 surface twisted with respect to a Brauer class, i.e. a torsion class $\alpha \in H^2(X, \mc O_X^*)$. 
In this setting there is a notion of  $\alpha$-twisted coherent sheaves and we denote by  $D^b(S,\alpha)$ the bounded derived category of $\alpha$-twisted coherent sheaves  on $S$ (see \cite[\S 1]{Huybrechts-Stellari}). 
Note that is the case for untwisted K3 surfaces, for any Brauer class the Serre functor on $D^b(S, \alpha)$ is $[2]$, the shift by $2$. This is a fundamental property which is the starting point of all the results of this area starting from \cite{Mukai84} and which will be shared by the Kuznetsov component of a cubic fourfold (to be introduced in the next section).

\subsection{The residual component of a cubic fourfold, and its Mukai lattice}

For more details on this section we refer  the reader to \cite{Macri-Stellari-survey, Kuznetsov-cubic,Kuznetsov-Calabi-Yau,Addington-Thomas,BLMNPS}.

In this section, $X$ will denote a smooth cubic fourfold. We let $D^b(X)$ be the bounded derived category of coherent sheaves on $X$. The Kuznetsov, or residual, component of $X$ is the admissible subcategory of $ D^b(X)$  defined as the left orthogonal
\[
\Ku(X):=\langle \mc O_X, \mc O_X(1), \mc O_X(2) \rangle^\perp \subset D^b(X)
\]
to the exceptional collection $\mc O_X, \mc O_X(1), \mc O_X(2)$. By  \cite{Kuznetsov-Calabi-Yau}, $\Ku(X)$ is a non-commutative K3 surface. This means that  it is an admissble subcategory of the bounded derived category of a smooth projective variety, that its Serre functor is the shift by $2$, and that its Hochschild (co)homology coincides with that of a K3 surface (see Definitions 2.21-22 and 2.31 of \cite{Macri-Stellari-survey}). 
The numerical Grothendieck group of $X$ is the quotient
\[
K_{num}(X):=K(D^b(X)) \slash \ker \chi,
\]
where $K(D^b(X))$ is the Grothendieck group of $X$ and $\ker \chi$ is the radical of the Euler pairing $\chi(E,F)=\sum (-1)^i \Hom_{D^b(X)}(E,F[i])$ (by Serre duality, the left and right radical of $\chi$ coincide). We say that $\Ku(X)$ is geometric if there is a K3 surface $S$ and an equivalence of categories $\Ku(X) \cong D^b(S)$. If this is the case, we say that $X$ has an associated K3 surface and that $S$ is the K3 surface associated to $X$ \cite{Kuznetsov-cubic,Addington-Thomas,BLMNPS}. This holds for example when $X$ is a Pfaffian cubic fourfold \cite{Kuznetsov-cubic}.

Let $K_{top}(X)$ be the topological K-theory of $X$ (see \cite{Atiyah-Hirzebruch, Addington-Thomas} as well as \cite[\S 3.4]{Macri-Stellari-survey} and references therein). Since $X$ has no odd cohomology, by \cite[\S 2.5]{Atiyah-Hirzebruch}, $K_{top}(X)$ is torsion free and it coincides with the Grothendieck group $K^0_{top}(X)$ of topological complex vector bundles on $X$. Moreover, the Mukai vector determines an injective morphism
\be
\begin{aligned} \label{MukaivecKtheory}
v: K_{top}(X) &  \longrightarrow H^*(X,\Q) \\
E &  \longmapsto v(E):=\ch(E) \sqrt{\td(X)}
\end{aligned}
\ee
which is an isomorphism when tensored by $\Q$. Define a bilinear pairing by setting
\[
(\alpha, \beta):=-\chi(\alpha, \beta):=-p_*(\alpha^\vee \otimes \beta) \in \Z
\]
where $p: X \to pt$ is the map to a point and $p_*$ is the push-forward in topological K-theory. When restricted to $K_{num}(X)$ this pairing agrees, up to a sign, with the Euler pairing $\chi$ considered above.
We set
\be \label{Ktheoryresidual}
\wt H(\Ku(X), \Z):=\{ x \, | \, (\mc O_X(i), x)=0, \,\, i=0,1,2\} \subset K_{top}(X)
\ee
By \cite[\S 2]{Addington-Thomas}, the pairing $( \cdot, \cdot )$ restricts to a symmetric bilinear pairing on $\wt H(\Ku(X),\Z)$, making this into a lattice, called the Mukai lattice of the Kuznetsov component of $X$. We define a weight $2$ Hodge structure on $\wt H(\Ku(X))=\wt H(\Ku(X),\Z) \otimes \mathbb C$ by setting
\[
\wt H^{2,0}(Ku(X))):=v^{-1}(H^{3,1}(X)),  \quad \wt H^{0,2}(\Ku(X))):=v^{-1}(H^{1,3}(X)),
\]
and $\wt H^{1,1}(\Ku(X)):=v^{-1}(\oplus H^{p,p}(X))$. Here, by abuse of notation we still denote by $v$ the induced isomorphism $K_{top}(X) \otimes \C   \to H^*(X,\C)$ as well as its restriction to $\wt H(\Ku(X), \Z) \otimes \mathbb C$.
We also set
\[
\wt H_{Hdg}(\Ku(X)):=\wt H(\Ku(X),\Z) \cap \wt H^{1,1}(\Ku(X)), \quad \wt H_{alg}(\Ku(X)):=K_{num}(\Ku(X)),
\]
where $K_{num}(\Ku(X))$ is the numerical Grothendieck group of $\Ku(X)$ (eg. see \cite[Def. 2.20]{Macri-Stellari-survey}).
By  \cite{Voisin-Some-Aspects}, the Hodge conjecture for cubic fourfolds holds, so $\wt H_{Hdg}(\Ku(X))=\wt H_{alg}(\Ku(X))$. As shown by Addington-Thomas, if $X$ has an associated K3 surface $S$, then $\wt H_{Hdg}(\Ku(X)) $ is the  Mukai lattice of $S$  \cite[Lem. 3.17]{Macri-Stellari-survey}.

These definitions can be adapted to families  $f: \mc X \to T$ of smooth cubic fourfolds. For the purpose of this note we will only need to consider the case when $T$ is a smooth quasi-projective curve, so this is what we  assume in what follows. Note, however, that these definitions can be adapted to a much more general context (see \cite[\S 5.1]{Kuznetsov-base-change} and \cite[\S 3.2 and 30.1]{BLMNPS}). We define
\[
\Ku(\mc X/T):=\langle \mc O_{\mc X}\otimes D^b(T), \mc O_{\mc X}(1)\otimes D^b(T), \mc O_{\mc X}(2)\otimes D^b(T) \rangle^\perp \subset D^b(\mc X)
\]
where for $i=0,1,2$ $\mc O_{\mc X}(i)$ denotes the corresponding relative line bundle on the family of cubic fourfolds. Here the usual semi-orthogonality condition between two $T$-linear  (i.e. closed under tensorization with objects  in $D^b(T)$) subcategories  $\mc A, \mc B \subset D^b(\mc X)$ is replaced by the condition $f_* R\Shhom(B, A)=0$ for all $A \in \mc A $ and $B \in \mc B$ (see \cite[Lemma 2.7]{Kuznetsov-base-change}). There is a notion of base change for derived categories, which we will not recall and instead refer to \cite[\S 5.1]{Kuznetsov-base-change}. With this notion of base change, for every closed point $t \in T$ the base change of $\Ku(\mc X/T)$ at $t$ is $\Ku(\mc X_t)$. 

Since topological $K$-theory is a topological invariant, the definition of Mukai lattice works in families, yielding a local system
\be \label{locsystemMukai}
\wt{\mc H}(\Ku(\mc X/T), \mathbb Z),
\ee
over $T$ whose fiber at a point $t \in T$ is $\wt H(\Ku(\mc X_t), \mathbb Z)$.
Deforming to a cubic fourfold whose Kuznetsov component is geometric shows that, as an abstract lattice, the Mukai lattice of the Kuznetsov component of a cubic fourfold is nothing but the extended K3 lattice.

A natural question to ask is for which cubic fourfolds is the Kuznetsov component geometric and, more generally, when is it equivalent to the derived category of a twisted K3 surface. The following fundamental theorem says that this is the case for countably many hypersurfaces in the moduli spaces of cubic fourfolds:

\begin{thm}[Addington-Thomas, Huybrechts, \cite{BLMNPS}] Let $X$ be a cubic fourfold. Then there exists a twisted K3 surface such that $\Ku(X) \cong D^b(S, \alpha)$ if and only if there exists a class $v \in \wt H^{1,1}(\Ku(X),\Z)$ with $v^2=0$.
\end{thm}
\begin{proof}
This is \cite[Prop. 33.1]{BLMNPS}.
\end{proof}

\subsection{Bridgeland stability conditions}

Next we give a very quick introduction to Bridgeland stability conditions, following \cite[4.1]{Macri-Stellari-survey} and mostly to fix notation. We refer the reader to loc. cit, as well as to \cite{Bayer, BLMS, Bridgeland-K3} for more thorough treatments.

As usual, let $\mc D$ be  the derived category of a (twisted) K3 surface or the Kuznetsov component of a cubic fourfold (most of what follows works more generally for a non-commutative smooth projective variety, but we won't pursue this generality here). Let $\Lambda$ be a free abelian group of finite rank and let $v: K(\mc D) \to \Lambda$ be a surjective homomorphism.

\begin{defin}A Bridgeland stability on $\mc D$ with respect to $\Lambda$ (and $v$) is a pair $\sigma=(\mc A, Z)$, where $\mc A \subset \mc D$ is the heart of a bounded $t$-structure and where $Z: \Lambda \to \mathbb C$ is a group homomorphism (called the central charge) such that:
\begin{enumerate}
\item $Z$ is a stability function, i.e. for every $0 \neq A \in\mc A$, $\Im Z(A) \ge 0$ and if $\Im Z(A)=0$, then $\Re Z(A) <0$.
\item When considering $\sigma$-semistability in $\mc A$, i.e. semistability with respect to the slope 
\[
\mu(\cdot):=\frac{-\Re Z(\cdot) }{ \Im Z(\cdot)}
\]
then every non-zero object in $\mc A$ has a Harder-Narasimhan filtration. This means that for every $0 \neq A \in \mc A$ there exists a finite sequence of injections in $\mc A$
\[
0=A_0 \to A_1 \dots A_{n-1} \to A_n=A
\]
with the property that the quotients $B_i=A_i / A_{i-1}$ are semistable with respect to $\mu(\cdot)$ and have decreasing slopes $\mu(B_1)>\mu(B_2)> \cdots >\mu(B_n)$.
\item $\sigma$ satisfies the support property, i.e. there exists a quadratic form $Q$ on $\Lambda \otimes \mathbb R$ such that $Q_{|\ker Z}$ is negative definite and $Q$ is non-negative on all $\sigma$-semistable objects.
\end{enumerate}
\end{defin}

We will mostly omit $v$ from the notation and, for example, write $Z(E)$ instead of $Z(v(E))$. 

The $\sigma$-semistable objects of $\mc D$ are declared to be the shifts of the $\sigma$-semistable objects of $\mc A$. The phase of a $\sigma$-semistable object $F=E[n]$, where $E \in \mc A$ is $\sigma$-semistable, is defined to be $\phi+n$, where the phase $\phi$ of $E \in \mc A$  is the unique $\phi \in [0,1)$ such that $Z(E) \in \mathbb R_{>0} \exp(\pi i \phi)$.

Let $\Stab_{\Lambda}(\mc D)$ be the set of Bridgeland stability conditions on $\mc D$ with respect to $\Lambda$ (and $v$).
A fundamental theorem of Bridgeland \cite{Bridgeland-Annals} endows $\Stab_{\Lambda}(\mc D) $ with a structure of complex manifold and proves that  the forgetful morphism $\Stab_{\Lambda}(\mc D)  \to \Hom(\Lambda, \mathbb C)$, $\sigma=(\mc A, Z) \mapsto Z$ is a local homeomorphism. In particular, the complex dimension of  $\Stab_{\Lambda}(\mc D) $ is equal to $\rk \Lambda$.

A stability condition is called numerical if $v$ factors via $K_{num}(\mc D)$, the numerical Grothendieck group of $\mc D$, and  full numerical if its numerical and the lattice $\Lambda $ is equal to $K_{num}(\mc D)$ (for the categories we are considering, $K_{num}(\mc D)$ is finitely generated). 
We denote by $\Stab(\mc D)$ the space of full numerical stability conditions.

Fixing a class $v \in \Lambda$ gives a wall and chamber structure on $\Stab_\Lambda(\mc D)$. More precisely, there is a locally finite set of $\mathbb R$-codimension one submanifolds called walls which have the following properties (see \cite[12.13]{BLMNPS}). Let $W$ be a finite intersection of walls. A chamber in $W$ is defined to be a connected component of the complement in $W$ of the union of the intersection of $W$ with other walls. As $\sigma$ varies in each chamber, the set of $\sigma$-semistable (and of $\sigma$-stable) objects of class $v$ is constant. Moreover, up to the natural action of $\wt{GL_2^+(\mathbb R)}$ on $\Stab_\Lambda(\mc D)$ (see \cite[Lemma 8.2]{Bridgeland-Annals}), every wall can be locally described by an equation of the form $\Im Z(u)=0$, for some $0 \neq u $.  Recall that $GL_2^+(\mathbb R)$ denotes the group of $2\times2$ matrices with positive determinant and that $\wt{GL_2^+(\mathbb R)}$ denotes its universal cover. 

Finally, in each chamber there is a $Z$ defined over $\mathbb Q[i]$ (i.e. such that $Z(\Lambda) \subset \mathbb Q[i]$).
A stability condition is called $v$-generic if it does not lie on any wall.

In certain cases, for example when considering families of stability conditions, it is natural to consider stability conditions with respect to different lattices. We won't delve into this, but simply notice that a numerical stability condition naturally induces a full numerical one.

In general, showing the existence of stability conditions for a given category is a very hard problem. For the derived category of a (twisted) K3 surface, respectively for the Kuznetov component of a smooth cubic fourfolds, the non emptiness of the space of full numerical stability conditions has been established in \cite{Bridgeland-K3, Huybrechts-Stellari-Stability}, respectively in \cite{BLMS}, by explicitly constructing Bridgeland stability conditions. These stability conditions all lie in the same connected component, which we denote by 
\[
\Stab^\dagger(\mc D).
\]
A fundamental breakthrough in this area was achieved in \cite{BLMNPS}, where the  notion of family of Bridgeland stability conditions is introduced. We won't attempt to give this subtle definition, and instead refer the reader to Definitions 20.5 and 21.15 of  loc. cit.

\subsection{Moduli stacks and (good) moduli spaces}

In this subsection we briefly recall the definition of moduli stack of objects in $\mc D$, where $\mc D$ is the derived category of a (twisted) K3 surface or the Kuznetsov component of a cubic fourfold. We follow \cite[\S 5.3]{Macri-Stellari-survey} and refer the reader to \cite{Lieblich, BLMNPS, Alper-HalpernLeistner-Heinloth} for details and a more thorough treatment of the topic. 

We let $X$ be a K3 surface or a cubic fourfold, so that $\mc D \subset \mc D^b(X)$ is an admissible subcategory. Following \cite[Def. 9.1]{BLMNPS} (see also \S 5.3 of \cite{Macri-Stellari-survey}), we consider the  functor from schemes to groupoids defined by 
\[
\begin{aligned}
\mathfrak{M}: Sch^{op} &\longrightarrow Gpds \\
B &\longmapsto \{E \in \mc D_{B-perf} \, | \, \Ext^i(E_b, E_b)=0  \,\, \forall i <0, \, \text{ for all geometric point } b \in B\}
\end{aligned}
\]
where $\mc D_{B-perf}$ is the intersection in $D_{qc}(X \times B)$ of the category  $D_{B-perf} (X × B)$ of $B$-perfect complexes on $X \times B$ with the smallest triangulated subcategory of $D_{qc}(X \times B)$ which contains $\mc D \boxtimes D^b(B)$ and which is closed under arbitrary direct sums (see Definition 5.9 of \cite{Macri-Stellari-survey} and following paragraph).
The second condition on the vanishing of the Ext-groups in negative degree defines what are called universally gluable complexes (see \cite[Def. 2.1.8, Prop]{Lieblich} or also \cite[Def 8.5]{BLMNPS}).
By \cite[Thm. 4.2.1]{Lieblich} and \cite[Prop. 9.2]{BLMNPS}, the above defines an algebraic stack $\mathfrak M$ locally of finite type over $\mathbb C$, with separated diagonal (see also \cite[Thm 5.10]{Macri-Stellari-survey}).

It is important to note that definition of the functor above can be adapted to work  also in families, so that one can consider relative moduli spaces.
Even though in Sections \ref{sect ample}, \ref{modspacesISV} and \ref{sect second cohomology}  we will indeed use relative moduli spaces, introducing the correct definitions would take us too far. We therefore instead refer the reader to \cite[\S 8-9]{BLMNPS}.

In practice, one wants to consider moduli spaces of Bridgeland semimstable objects in $\mc D$ of a given numerical class. One would like to know that  the substack of $\mathfrak M$ parametrizing semistable objects of a given class is a stack of finite type over $\mathbb C$ and that it admits a good moduli space. As first observed by Toda  \cite{Toda}, the first condition holds if the semistable locus defines an open substack and if the set of semistable objects in a given class forms a bounded family. For K3 surfaces (more precisely, for those stability conditions in the main component) these two conditions were proved in \cite{Toda}. In \cite{BLMNPS}, these two properties are required to hold as part of the definition of  families of stability conditions, and are proved to hold for the stability conditions on the Kuznetsov component of a cubic fourfold that are constructed in \cite{BLMS}.

We now introduce the moduli stack of semistable objects. Fix a class $v \in \Lambda$ and a stability condition $\sigma=(\mc A, Z) \in \Stab_\Lambda(\mc D)$. We also fix the phase of the semistable objects we will consider, that is we fix a  $\phi \in \mathbb R_{>0}$ such that $Z(v) \in R_{>0} \exp (i\pi \phi)$. We let
\be
\mathfrak M_v (\mc D,\sigma, \phi)
\ee
be the substack of $\mathfrak M$ parametrizing semistable objects of class $v$ and phase $\phi$. See also Definition 21.11 of \cite{BLMNPS}.

\begin{rem} \label{pmv}  Clearly,  for every $n \in \Z$, there is a natural equivalence
\[
\mathfrak M_v (\mc D,\sigma, \phi) \to \mathfrak M_{(-1)^nv}  (\mc D,\sigma, \phi+n),
\]
given by the shift by $n$, so we will often drop the phase from the notation and write $\mathfrak M_v (\mc D,\sigma)$ instead. This means, in particular, that when studying properties of these moduli stacks and their (good) moduli spaces, one can always pass from $v$ to $-v$. If the category $\mc D$ is clear from the context, we will also simply write $\mathfrak M_v (\sigma)$.
\end{rem}

Using foundational results of Leiblich \cite{Lieblich} and Toda \cite{Toda}, one gets algebraicity of these stacks 

\begin{thm}(\cite{Toda,BLMNPS}) The stack
$\mathfrak M_v (\mc D,\sigma, \phi)$ is bounded and it is an open substack of $ \mathfrak M$. As a consequence, $\mathfrak M_v (\mc D,\sigma, \phi)$ is an algebraic stack of finite type over $\mathbb C$ with affine diagonal.
\end{thm}
\begin{proof} The case of K3 surfaces is one of the main result of \cite{Toda} (see Theorems 3.20, 4.5, Lemma 4.7 and Prop. 4.11). In the case of the Kuznetsov component of a cubic fourfold, this is proved in \cite{BLMNPS}: openness and boundedness are requested to hold as part of the definition of a family of stability conditions ( see Definition 20.5 and 21.15 of loc. cit), and as shown in Proposition 30.5 of loc.cit, the stability conditions constructed in \cite{BLMS} on the Kuznetsov component of a cubic fourfold satisfy these properties.  The fact that these two conditions guarantee the second statement is proved in \cite[Lem. 3.4]{Toda}.
The fact that the diagonal is affine follows from \cite[\href{https://stacks.math.columbia.edu/tag/0DPW}{Tag 0DPW}]{StacksProject}. 
\end{proof}

Next we recall the notion of good moduli space, which was introduced by Alper in \cite{Alper-gms} as a generalization of the notions of good quotient in GIT and of coarse moduli space.

\begin{defin} Let $\mathfrak N \to T$ be an algebraic stack over an algebraic space. A good moduli space for $\mathfrak N$ is an algebraic space $N$ over $T$ with a morphism $\pi: \mathfrak N \to N$ such that:
\begin{enumerate}
\item The push-forward functor $\pi_*$ on quasi-coherent sheaves is exact.
\item The natural map $\mc O_{N} \to \pi_* \mc O_{\mathfrak N}$ is an isomorphism.
\end{enumerate}
\end{defin}

If a good moduli space  for $\mathfrak N$ exists, we say that $\mathfrak N$ admits a good moduli space. 
Building on \cite[\S 7]{Alper-HalpernLeistner-Heinloth}, in \cite{BLMNPS} the existence of  proper good moduli spaces is established

\begin{thm} \label{thm gms}
Let $\mc D$ be  the derived category of a  (twisted) K3 surface or the Kuznetsov component of a cubic fourfold, let $v \in K(\mc D)$ be a Mukai vector, let $\sigma \in \Stab^\dagger(\mc D)$ a stability condition and $\phi \in \mathbb R$ a phase (compatible with $v$ and $\sigma$). 
\begin{enumerate}
\item The closed points of the moduli stack  $\mathfrak M_v (\mc D,\sigma, \phi)=\mathfrak M_v (\mc D,\sigma)$ are in bijection with the $\sigma$-polystable objects of class $v$. 
\item The moduli stack $\mathfrak M_v (\mc D,\sigma)$ admits a good moduli space
\[
\pi: \mathfrak M_v (\mc D,\sigma) \to  M_v (\mc D,\sigma, \phi)=M_v (\mc D,\sigma)= M_v(\sigma)
\]
which is proper over $\mathbb C$. 
\item The stable locus $M_v (\mc D,\sigma^s)$ of the moduli space admits a quasi-universal family (see Definition A.4 of \cite{Mukai-Tata}).
\end{enumerate}
\end{thm}
\begin{proof}
The first statement follows Lemmas 3.24, 3.25, Corollary 3.13 and Example 7.26 of \cite{Alper-HalpernLeistner-Heinloth}. The second statement is \cite[Thm 21.24]{BLMNPS}. For the third statement, we argue as follows. Using the Glueing Lemma \cite[Thm. 3.2.4]{BBD} and the same argument as in the proof of Theorem A.4  of \cite{Mukai-Tata}, it is enough to show that there is a universal family on an \'etale cover  of the moduli space. 
The existence of a universal family on an \'etale cover can be seen in the following way: by  \cite[Cor. 4.3.3]{Lieblich} the \'etale sheafifation of the rigidification (cf.  \cite[\S]{Abramovich-Corti-Vistoli} or also \cite[\S 6.2.8]{Alper-Notes-Stacks}) of the moduli functor is represented by an algebraic space, which is nothing but  the good moduli space $ M^s_\sigma(\mc D, v)$, and the natural morphism $\mathfrak M_v (\mc D,\sigma) \to  M_v (\mc D,\sigma) $ is a $\mathbb G_m$-gerbe. 
By definition of \'etale sheafification, there is an \'etale cover $U \to M_v(\mc D, \sigma)$ which lifts to a morphism $ U \to \mathfrak M_v(\mc D, \sigma^s)$.
\end{proof}

Another crucial ingredient for studying these moduli spaces are the numerical divisor class $\ell_\sigma$ on the moduli stack and the moduli space, which were first introduced by Bayer-Macr\`i  \cite[Thm 4.1]{Bayer-Macri-Proj} as a generalization of the determinantal line bundles of Le Potier.

\begin{thm} \label{thm ellsigma}
For every $\sigma=(\mc A, Z) \in \Stab(\mc D)$ there exist a numerical $\mathbb R$-divisor class $[\ell_{\sigma,\frak M_v(\sigma)}] \in N^1(\frak M_v(\sigma))$ which descends to a strictly nef $\mathbb R$-divisor class $[\ell_{\sigma_0, M_v(\sigma)}]$ on the good moduli space $M_v(\sigma)$. If $Z$ is defined over $\Q[i]$, then $[\ell_{\sigma_0, M_v(\sigma)}]$ is the class of a $\Q$-Cartier divisor on $M_v(\sigma)$.
\end{thm}
\begin{proof}
This is \cite[Thm. 21.25]{BLMNPS}, building on \cite[\S 3-4]{Bayer-Macri-Proj}.
\end{proof}

Note that the reference  \cite[Thm. 21.25]{BLMNPS} deals also with the case of relative moduli spaces. The fact that in certain circumstances this class can be defined for relative moduli spaces will be important in Section \ref{sect ample}.

\subsection{Results on K3 surfaces}
Let $(S, \alpha)$ be a twisted K3 surface and let $\mc D=D^b(S, \alpha)$.  Consider a  $B$-field  $B \in H^2(S,\mathbb Q)$ lifting $\alpha$. Recall that a $B$-field is defined as follows: since $H^3(S, \mathbb Z)=0$, there is a lift $B \in H^2(S, \mc O_S)$ of 
$\alpha$ under the map $H^2(S, \mc O_S) \to H^2(S, \mc O_S^*)$ induced by the exponential exact sequence. Since by definition $\alpha$ is torsion, $B \in H^2(S, \Q)$. We set  $H^*_{alg}(S,B, \Z) := (\exp(B) H^*_{alg}(X,\mathbb Q)) \cap H^*_{alg}(S, \Z)\subset H^*_{alg}(S, \Z)$ (see  \cite[\S 1]{Huybrechts-Stellari}). Note that the twisted Chern characters $\alpha$-twisted coherent sheaves lie in $H^*_{alg}(S,B, \Z)$ \cite[Prop. 1.2 and Rem 1.3]{Huybrechts-Stellari}.
We fix $v_0 \in H^*_{alg}(S, B, \Z)$ a primitive Mukai vector  with $v_0^2 \ge 2$ and for an integer $m \ge 2$ we set $v=mv_0$.

We collect in the following theorem a series of results of Bayer-Macr\`i on moduli spaces of Bridgeland semistable objects in $\mc D$ with Mukai vector $v$.

\begin{thm} \label{Projnongeneric} Let $v_0 \in H^*_{alg}(S, B, \Z)$ be a Mukai vector with $v_0^2 \ge 2$, let $m  \ge 2 $ be an integer, set $v=mv_0$, and let $\sigma \in \Stab^\dagger(\mc D)$ be a stability condition that is $v$-generic.

\begin{enumerate}
\item Then $M_v(\sigma)$ is a non-empty, normal irreducible symplectic variety of dimension $v^2+2$. The line bundle $\ell_\sigma$ is ample;
\item If $\tau \in  \Stab^\dagger(\mc D)$ is also $v$-generic, then $M_v(\sigma)$ and $M_v(\tau)$ are birational.
\item If $(m,v_0^2) \neq (2,2)$, so $v=mv_0$ is not of OG10 type, then $M_v(\sigma)$ is factorial with terminal singularities.
\end{enumerate}

\end{thm}

\begin{proof}
(1) Non-emptiness is \cite[Thm 2.15 ]{Bayer-Macri-MMP}; the irreducibility and the ampleness of $\ell_\sigma$ is \cite[Thm 1.3]{Bayer-Macri-Proj}; the fact that the moduli spaces are normal with symplectic singularities is \cite[Thm 3.10]{Bayer-Macri-MMP}.

(3) This is \cite[Thm 1.3 (a)]{Bayer-Macri-Proj}; Note that OG10-type moduli spaces admit a divisorial symplectic resolution, so the singularities are canonical, but not terminal.

(2) Requires a proof, which we give building on \cite{Bayer-Macri-MMP, Meachan-Zhang}. Since $ \Stab^\dagger(S)$ is connected and the walls are $\mathbb R$-codimension $1$ hypersurfaces which are locally finite, it is enough to consider the case when $\sigma^+:=\sigma$ and $\sigma^-:=\tau$ belong to adjacent chambers, i.e., two chambers that are separated by a wall $\mc W$.  Let $\sigma_0$ be a general stability condition on $\mc W$ (i.e.  $\sigma_0$ does not belong to any other wall).  If $\mc W$ is not a totally semistable wall (cf. \cite[Def. 2.20]{Bayer-Macri-MMP}) , then by definition there are $\sigma_0$-stable objects of class $v$. By openness of semistability, the natural morphism of moduli spaces $\phi^\pm: M_v(\sigma^\pm) \to M_v(\sigma_0)$, induced by the open immersion of stacks $\Mf_v(\sigma^\pm) \subset \Mf_v(\sigma_0)$ thus admits a birational inverse.

If $\mc W$ is a totally semistable wall, then $\dim M_v(\sigma_0)$ may jump, so $\phi^\pm$ may not be birational. 
In \cite[Prop. 5.2]{Meachan-Zhang}, the authors use the techniques developed by \cite{Bayer-Macri-MMP} to prove that the morphisms $\phi^\pm$ are still birational onto their image. For the readers sake, we recall  their argument. 
Let $\mc H$ be the rank $2$ lattice associated to $\mc W$ as in \cite[Prop. 5.1]{Bayer-Macri-MMP}. 
Since  $v$ is not primitive, the case where $\mc H$ contains an isotropic class pairing $1$ with $v$ is excluded. Hence by \cite[Thm 5.1]{Meachan-Zhang} there exists a spherical class $s \in \mc H$ such that $s \cdot v <0$ (cf.\cite[Thm 5.7]{Bayer-Macri-MMP}).
 In this case,  by \cite[Prop. 6.8]{Bayer-Macri-MMP} (see also \cite[Lem 6.3]{Meachan-Zhang}) there exists a Mukai vector $v_{min} \in \mc H$ (called the minimal class) and two autoequivalences $\Phi^\pm$ of $\mc D$ which have the following properties:
 
 \begin{enumerate}
\item $M_{v_{min}}(\sigma_0)^s \neq \emptyset$;
\item $\Phi^\pm(v_{min})=v$;
\item for every $\sigma_0$-stable object $E$ of class $v_{min}$, $\Phi^\pm(E)$ is $\sigma^\pm$-semistable of class $v$. 
\end{enumerate}
 
 Moreover, by the proof of \cite[Cor. 7.3]{Bayer-Macri-MMP} (see also \cite[Lem 6.3]{Meachan-Zhang}), the general $\sigma_0$-stable object $E'$ of class $v_{min}$ is such that $\Phi^\pm(E)$ and $\Phi^\pm(E')$ are not $S$-equivalent with respect to $\sigma_0$.
 
This means that, setting $\eta=\sigma^\pm$ and $ \sigma_0$,  the morphisms $M_{v_{min}}(\eta) \to  M_{v}(\eta)$ induced by $\Phi^\pm$ are birational onto their image and hence that $M_v(\sigma^\pm) \to M_{v}(\sigma_0)$ is also birational onto its image.

 In particular, $M_{v}(\sigma^{+})$ and $M_{v}(\sigma^{-})$ are birational.
\end{proof}

\section{Local structure and applications} \label{sect local}

Let $\mc D $ be $Ku(X)$ or $\mc D^b(S, \alpha)$.
In this section we study the local structure of the moduli space $M_v(\sigma)$ at a strictly polystable object $F=\oplus F_i$.  While in the beginning of the section we don't make any assumptions on $v$ and $\sigma$, for the factoriality result (Theorem \ref{locally fact})  we need to assume that $v=mv_0$, $v_0^2\ge 2$ is a non-primitive vector which is not of OG$10$-type and that $\sigma$ is $v$-generic. Otherwise, the result is not necessarily true. The key tools are Luna's \'etale slice theorem for stacks \cite{AHR},  the use of the Ext-quiver \cite{Kaledin-Lehn-Sorger, AS, ASII}, and ideas that go back to \cite{Drezet-Narasimhan, Drezet, Kaledin-Lehn-Sorger}.

We start by recalling the \'etale slice theorem for the stack $ \mathfrak M_v(\sigma)$, together with some refinements \cite{AHR, Alper-HalpernLeistner-Heinloth}.
For the basic definitions and results on stacks we refer the reader to \cite{Alper-Notes-Stacks, StacksProject}.
For an algebraic stack $\mc X$, we denote by $I_{\mc X}$ its inertia stack \cite[Def. 3.2.9]{Alper-Notes-Stacks}.

\begin{thm}(Luna's \'etale slice for $\mathfrak M_v(\sigma)$ \cite{AHR}) \label{luna} \label{propalper}  
Let $ F$ be a $\sigma$-polystable object of class $v$ defining a closed point $[F]  \in \mathfrak M_v(\sigma)$ and let $G$ be its automorphism group, which is linearly reductive. Then
\begin{enumerate}
\item  there exists an affine $G$-scheme $U$, a $G$-fixed point $u_0 \in U$, and a pointed \'etale morphism of stacks
\[
\epsilon: ([U \slash G],u_0) \to (\mathfrak M_v(\sigma), [F])
\]
which induces an isomorphism of stabilizers at $u_0$;
\item up to restricting $U$, we can assume that $\epsilon$ is affine, that it induces an isomorphism of stabilizers (in the sense that  the natural map $I_{[U \slash G]} \to [U \slash G] \times_{\mathfrak M_v(\sigma)} I_{\mathfrak M_v(\sigma)}$ is an isomorphism), and finally, that $\epsilon$ sends closed points to closed points;
\item the morphism $\epsilon$ induces a Cartesian diagram
\[
\xymatrix{
[U \slash G] \ar[d] \ar[r]^\epsilon &  \mathfrak M_v(\sigma) \ar[d]\\
U \sslash G\ar[r] &   M_v(\sigma) 
}
\]
where $\pi:  \mathfrak M_v(\sigma)  \to   M_v(\sigma) $ is the good moduli space.
\end{enumerate}
\end{thm}

\begin{proof}
The existence of an affine $G$-scheme $U$ with the required properties is the content of \cite[Thm 1.2]{AHR}. By Theorem \ref{thm gms}  the moduli stack $ \mathfrak M_v(\sigma)$ is an algebraic stack of finite type with { affine diagonal} which admits a good moduli space.  The fact that $\epsilon$ is affine is proved in \cite[Prop. 3.2]{AHR}. 
By  \cite[Thm A]{Alper-HalpernLeistner-Heinloth} it is $\Theta$-reductive and $S$-complete \cite[Thm A]{Alper-HalpernLeistner-Heinloth}. In particular, it has unpunctured inertia \cite[Thm 5.3]{Alper-HalpernLeistner-Heinloth}. Thus we can use \cite[Prop. 4.3]{Alper-HalpernLeistner-Heinloth}  and  \cite[Lem. 3.28]{Alper-HalpernLeistner-Heinloth} to deduce that the natural map  $I_{[U \slash G]} \to [U \slash G] \times_{\mathfrak M_v(\sigma)} I_{\mathfrak M_v(\sigma)}$ is an isomorphism  and that $\epsilon$ sends closed points to closed points.
The fact that the diagram above is cartesian is the content of \cite[Theorem 4.12]{AHR}.
\end{proof}

While this theorem applies to more general stacks (see \cite{AHR}), in the case of moduli space of objects on K3 categories, the Ext-quiver associated to a polystable object can be especially useful for the local study of the moduli space at a point $[F]$. It was introduced implicitly in \cite{Kaledin-Lehn-Sorger} and explicitly in \cite{AS}.

\begin{defin} \cite[Prop. 6.1]{AS}
Let $\sigma \in \Stab^\dagger(\mc D)$ be a Bridgeland stability condition and let $F$ be a $\sigma$-polystable object. We let $F=\oplus_{i=1}^k F_i^{m_i}$ be its decomposition in distinct stable factors.  The Ext-quiver $Q=(E,A)$ associated to $F$ is the  quiver $Q$ with vertex set $E=\{1, \dots, k\}$ and with  $\ext^1(F_i,F_j)$ arrows from $i$ to $j$ if $i < j$, and $\ext^1(F_i,F_i)/2$ loops at each vertex. The double quiver $\bar Q$ is the quiver obtained from $Q$ by adding to $Q$ a new arrow with the reversed orientation for every arrow of $Q$.
\end{defin}

Note that the double quiver does not depend on the choice of the ordering of the vertices. Set $\mathbf m=(m_1, \dots, m_k)$,
so that 
\[
\Aut(F)=GL(\mathbf m)=\prod GL(m_i),  \quad  \Ext^1(F,F)=Rep(\bar Q,\mathbf m) , \quad \Ext^2(F,F)=\frak{Lie}(\Aut(F))= \mathfrak{gl}(\mathbf m).
\]
Recall that for a quiver $Q$ and dimension vector $\mathbf m=(m_1, \dots, m_k)$,
\[
Rep(\bar Q,\mathbf m) = \oplus_{a \in A} \Hom (V_{s(a)}, V_{t(a)}),
\]
where for every arrow $a \in A$ we let $s(a)$ (resp. $t(a)) \in E$ denote the source (resp. the target)  of $a$ and where for every $i \in E$, $V_i$ is a vector space of dimension $m_i$.
Under these identifications, the moment map 
\[
\mu:\Ext^1(F,F) \to \Ext^2(F,F)=Lie(\Aut(F))
\]
 is the Yoneda product, which is also the quadratic part of the formal Kuranishi map \cite{Kaledin-Lehn-Sorger, Lieblich}. The quiver variety  is defined as the symplectic reduction
 \[
 N_Q(\mathbf m):=\mu^{-1}(0) \sslash G.
 \]
 The object $F$ is said to satisfy the formality property if the dg algebra $R\Hom(F,F)$ is formal, i.e., if it is quasi-isomorphic (as dg algebra) to its cohomology algebra \cite[\S 3.1]{Kaledin-Lehn}.
 If $F$ satisfies the formality property, then  the formal completion of $\mu^{-1}(0)$ in $0$ is isomorphic to the zero locus of the formal Kuranishi map, and hence it is a formal miniversal space for the deformation functor of $F$. As a consequence, the moduli space is locally (analytically or in the \'etale topology) isomorphic to $ N_Q(\mathbf m)$. For more on this see \cite{ASII} and references therein. While the formality property is known to hold in many situations (see \cite{Budur-Zhang, Bandiera-Manetti-Meazzini,ASII}), we have chosen to use arguments that do not require it (though the formality property would simplify some of the arguments in the proof of factoriality).

From now on we assume that the Mukai vector $v=mv_0$, $m \ge 2$, $v_0^2 \ge 2$ is not of OG$10$-type and that the stability condition $\sigma$ is $v$-generic. The following Proposition reformulates   \cite[Prop. 3.8-3.11]{Kaledin-Lehn-Sorger} to the case of more general K3 categories. Recall that a variety is called factorial if every Weil divisor is Cartier. This holds if and only if the variety is (Zariski) locally factorial, meaning that all of its local rings are unique factorization domains.

\begin{prop}  \label{Ufac} Let $F$ be a $\sigma$-polystable object of Mukai vector $v$ and $U$ be as in Theorem \ref{luna}. Then $U$ is a normal local complete intersection regular in codimension $\ge 3$. Moreover, up to passing to a $G$-invariant  open saturated neighborhood of $u_0$, we can assume that $U$ is factorial.
\end{prop}
\begin{proof}
The proof of the first statement is the same as that of Proposition 3.8 of \cite{Kaledin-Lehn-Sorger}. The proof of the second statement follows, using the first statement and Grothendieck's factoriality theorem for complete intersections which are regular in codimension $3$, as in the proof of Proposition 3.11 of \cite{Kaledin-Lehn-Sorger}.
\end{proof}

\begin{cor} \label{cor irr}
The moduli space $M_v$ is normal. In particular, its connected components are irreducible.
\end{cor}

We now recall the natural stratification of $M_v(\sigma)$ by polystable type. Let $F=\oplus_{i=1}^k F_i^{\alpha_i}$ be a $\sigma$-polystable object of class $v$. We assume that  the $F_i$ are $\sigma$-stable of class $v_i:=m_i v_0$ and that $F_i \ncong F_j$ for $ i \neq j$. Note that $m=\sum \alpha_i m_i$. The sequence of integers
\[
\underline \alpha:=(\alpha_1, m_1, \dots,\alpha_k, m_k  ),
\]
which we will refer to as the \emph{polystable type of $F$},
determines the class and the multiplicities of the decomposition of $F$ in direct sum of stable objects.
For every  polystable type $\underline \alpha= (\alpha_1, m_1, \dots,\alpha_k, m_k  )$,  let $P_{\underline \alpha} \subset M_v$ be the set of points corresponding to polystable objects of type $\underline \alpha$ and consider the finite morphism
\be  \label{stratificationmorph}
\begin{aligned}
 \Phi_{\underline \alpha}: \prod M_{v_i} & \longrightarrow M_v.\\
 \{F_{i}\} &\longmapsto  \oplus_{i=1}^k F_{i}^{\oplus \alpha_i}
\end{aligned}
\ee
Then $P_{\underline \alpha} $ is the image under  $\Phi_{\underline \alpha}$ of the open subset of $\prod M_{ v_i}$ parametrizing distinct stable objects and hence it is a locally closed subset of $M_v(\sigma)$. For example, $P_{(1,m)} $ is the stable locus and $P_{(m,1)} = M_{v_0}(\sigma)$ parametrizes objects of the form $F_0^m$, with $F_0 \in M_{v_0}(\sigma)$.
As $\underline \alpha$ varies among the possible polystable types of objects of class $\sigma$, the $P_{\underline \alpha}$ define a stratification of $M_v(\sigma)$.

\begin{rem} \label{deepeststratum}
Note that since for every $m$, there is a closed embedding $M_{v_0} \to M_{mv_0}$, $F \mapsto F^{\oplus m}$, it follows that $P_{(m,1)} = M_{v_0}(\sigma)$ is the deepest stratum in the sense that it is contained in the closure of every other stratum.
\end{rem}


\section{Ampleness of $\ell_\sigma$} \label{sect ample}


In this section we prove that for a generic stability condition the moduli spaces of Bridgeland stable objects in the Kuznestov component of a cubic fourfold, with non-primitive Mukai vector  that is not of OG$10$-type,  
are irreducible of dimension $v^2+2$ and that the numerical class $\ell_\sigma$  of Theorem \ref{thm ellsigma} is ample. This is done in Theorem \ref{thmproj}. 

Let $X$ be a cubic fourfold. Let $\mc D$ be its Kuznetsov component. Fix a primitive Mukai vector $v_0 \in K_{num}(\mc D)$ with $v_0^2 \ge 2$, an integer $m \ge 2$ and set $v=mv_0$. Assume that $v$ is not of OG$10$-type.

\begin{thm} \label{thmproj} Let $X$ and $v$ be as above. For any $v$-generic $\sigma \in \Stab^\dagger (X)$, let $M_v( \sigma)=M_v(\mc D, \sigma)$ be the good moduli space of the moduli stack $\mathfrak M_v(\mc D, \sigma)$, which exists by Theorem \ref{thm gms}. Then $M_v( \sigma)$ is an irreducible projective symplectic variety of dimension $v^2+2$ with terminal singularities and the numerical class $\ell_\sigma$ is ample.
\end{thm}

The case when $v$ is  primitive case is proved in \cite{BLMNPS}, while the analogue result for when $v$ is of OG$10$-type ($m=2$ and $v_0^2=2$) is proved in \cite{Li-Pertusi-Zhao} (in this case, the statement is that the moduli space is irreducible projective with canonical singularities and that it admits a symplectic resolution which is an irreducible holomorphic symplectic manifold of OG$10$-type.)
The original argument of  \cite{BLMNPS}, which is also used in  \cite{Li-Pertusi-Zhao}, does not extend to the case when there is no symplectic resolution. Here we use the argument of the revised version of \cite{BLMNPS}, which was suggested by the author of this note.

\begin{rem}
Thanks to the recent \cite{Villalobos} and the results on the local structure of moduli spaces, one can conclude that these moduli spaces are always projective. The following argument shows a little more, namely that the class $\ell_\sigma$ is ample. \end{rem}

\begin{proof} 
We follow the strategy of \cite[\S 29-30]{BLMNPS}.
Let $\mc X \to C$ be a  non-isotrivial family of smooth cubic fourfolds parametrized by a smooth connected quasi-projective curve $C$, with the property that $X=\mc X_{c_0}$ for some point $c_0 \in C$ and that $v$ extends to a section of the local system of Mukai lattices $\wt {\mc H(}Ku(\mc X_c), \mathbb Z)$ which stays algebraic for every $c \in C$. This can be achieved by choosing $C$ inside the Hassett divisor corresponding to the degree $4$ component of $v$.
  As on pages 315-316 of \cite{BLMNPS}, we can find such a family such that there are infinitely many points $c \in C$ for which  $Ku(\mc X_c) \cong D^b(S_c, \alpha_c)$, for some twisted  K3 surface $(S_c, \alpha_c)$. Moreover, by Lemma 32.5 of loc. cit, we can assume that the derived equivalence sends $\Stab^\dagger( Ku(\mc X_c))$ to  $\Stab^\dagger(S_c, \alpha_c)$.
  
Up to passing to a finite cover, we can assume that $\mc X$ comes with a family of lines which are not contained in any plane and that the fixed lattice of the monodromy action on $\wt H(Ku(\mc X_{c_0}), \mathbb Z)$ is the lattice $M:=\wt H_{Hdg}(Ku(X_c))$, where $c \in C$ is very general \cite[Thm. 4.1]{Voisin-Handbook}. 
Then, by \cite[Prop. 30.5]{BLMNPS} there exists a family of stability conditions on $Ku(\mc X)$ over $C$ and with respect to the dual lattice $M^\vee$ (for the definition of family of stability conditions see Definition 21.15  of loc. cit.). 
Denote it by  $\famsigma \in \Stab_{M^\vee}(Ku(\mc X)/C)$. 
By \cite[Rem 30.6.]{BLMNPS}, we can choose it so that $\famsigma_{c_0}=\sigma \in \Stab^\dagger(Ku(X))$. 
As in the proof of \cite[Prop. 30.8]{BLMNPS}, by  the openness of geometric stability and the boundedness of relative moduli spaces, $\sigma_{c}$ is $v$-generic except possibly for finitely many $c \in C$.  
Up to passing to an open neighborhood $c_0 \in C$, which still contains infinitely many points where  $Ku(X_c)$ is geometric, we can thus assume that $\famsigma_c$ is $v$-generic for all $c \in C$.
 By \cite[Thm 21.24 (3)]{BLMNPS}, there exists a good moduli space,
\[
\pi: M_v(\mc D,\famsigma) \to C
\]
which is an algebraic space proper over $C$ and such that $M_v(\mc D,\famsigma)_c=M_v(\mc D_c,\sigma_c)$. 
Using \cite[Lem 30.4 and Prop. 30.8 (3)]{BLMNPS}, we can apply \cite[Thm 21.25]{BLMNPS}  and  conclude that there exists a real numerical class $[\ell_\famsigma]$ on $M_v(\mc D,\famsigma)$, which is strictly nef on each fiber. Up to slightly moving $\famsigma$ we can assume that this class corresponds to an actual line bundle $\ell_\famsigma$.

By Corollary \ref{cor irr}, all the fibers of $\pi$ are normal, in particular their connected components are irreducible. 
By Theorem \ref{Projnongeneric} the fibers of $\pi$ over the points where the corresponding Kuznetsov component is geometric are irreducible; as a consequence, there is a unique irreducible component $M' \subset  M_v(\mc D,\famsigma)$  dominating $C$.
Moreover, since $C$ is normal, $M' \to C$ has connected and thus irreducible fibers. 
By Mukai's generalized theorem \cite[Thm 31.1]{BLMNPS} (cf. also \cite{Perry-Mukai}), the restriction of $\pi$ to the open subset parametrizing stable objects is smooth, so, since $M'$ contains the stable locus of some fibers, it contains the stable locus of every fiber. In particular, any extra component of $M_v(\mc D,\famsigma)$ is entirely contained in the strictly semistable locus and is mapped to a point in $T$.

Fix a $c \in C$.  We  show by induction on $m$ that $M_{mv_0}(\sigma_c)$ is irreducible.  For $m=1$, this is  \cite[Thm 29.2]{BLMNPS}. 
Assuming the result is known for all $m'<m$, it follows that for every polystable type $\ua \neq (1,m)$, $P_{\ua}$ is irreducible. 
Since every $P_{\ua}$ contains $P_{(m,1)}$ in its closure (cf. Remark \ref{deepeststratum}), it follows by induction that the polystable locus of $M_{mv_0}(\sigma_c)$ (which is precisely the singular locus of $\Sing (M_{mv_0}(\sigma_c))$) is connected. 
Since $\Sing (M_{mv_0}(\sigma_c)) \cap M'_c$ is non-empty for every $c$ (indeed it is non empty for countably dense $c \in C$), it follows that $\Sing (M_{mv_0}(\sigma_c)) \subset M'_c$. Thus $M_{mv_0}(\sigma_c)=M'_c$ is irreducible of dimension $v^2+2$.

By \cite[Thm 4.3]{Kuznetsov-Markushevich} the stable locus $M_{mv_0}(\sigma_c)^s$ has a holomorphic symplectic form, so the irreducible algebraic space $M_{mv_0}(\sigma_c)$ is $K$-trivial. By  Corollary \ref{cordimsing} and the main result of \cite{Flenner} the holomorphic symplectic form on $M_{mv_0}(\sigma_c)$ extends to a holomorphic $2$-form on any resolution of singularities, i.e. $M_{mv_0}(\sigma_c)$ is a symplectic variety in the sense of Beauville (see Definition \ref{defin symplsing} below). Since by Corollary \ref{cordimsing} below, the codimension of the singular locus of $M_{mv_0}(\sigma_c)$ is greater or equal to $4$, by \cite[Cor 1]{Namikawa-Note} (see also Lemma \ref{lemterminal4} below), $M_{mv_0}(\sigma_c)$ has terminal singularities. By \cite{Kawamata-Def}, the total space $M_v(\mc D,\famsigma)$, which is normal and connected,  has canonical singularities.

Since by Theorem \ref{Projnongeneric} $\ell_\famsigma$ is relatively ample over a non-empty open subset of $C$, it is big over $C$. Since as we have observed the algebraic space $M_v(\mc D,\famsigma)$ has canonical singularities, by the base-point-free Theorem (see \cite[Thm 3.3 and Rem pg 162]{Ancona} for the result in the context of algebraic spaces) and the fact that $\ell_\famsigma$ is strictly nef, it follows that $\ell_\famsigma$ is ample over $C$.
\end{proof}

\begin{cor} \label{cordimsing} The  singular locus of $M_{mv_0}(\sigma)$ has codimension $\ge 4$.
\end{cor}
\begin{proof} This is  the same computation as \cite[Prop. 6.1]{Kaledin-Lehn-Sorger}:
Since $(m,v_0^2) \neq (2,2)$, for any decomposition $m=\sum m_i$, we have $(mv_0^2)+2 \ge \sum (m_i^2v_0^2+2)+4$, so the codimension of each stratum of the singular locus is at least $4$.
\end{proof}

 \section{Factoriality}
 
 We now turn to the  factoriality of the moduli spaces. 
The factoriality of Gieseker moduli spaces of sheaves on K3 surfaces with non primitive Mukai vector not of OG10-type was proved in \cite{Kaledin-Lehn-Sorger}, following ideas of \cite{Drezet, Drezet-Narasimhan}. We  adapt  their arguments to prove the result  for these moduli spaces, which are not known to be global quotients. Note that the factoriality does not necessarily hold {analytically} locally.
Recall that a normal variety is called factorial if every Weil divisor is Cartier. This holds if and only if it is locally factorial, in the sense that all of its local rings are unique factorization domains (UFD).

 We start with some preliminary results.
The following is an adaptation to this context of an argument of Dr\'ezet-Narasimhan \cite{Drezet-Narasimhan, Drezet}.

\begin{prop} \label{corstabslice} 
Let $F \in \mathfrak M_v(\sigma)$ be a $\sigma$-polystable object with Mukai vector $v$ as above, let $U$ be the corresponding slice as in Theorem \ref{luna} above, and set $G=\Aut(F)$.
Up to passing to a  $G$-invariant saturated open neighborhood of $u_0 \in U$ we can assume that for every $u \in U$ with closed orbit there exists an irreducible closed subvariety $Z=Z_u$ containing both $u$ and  $u_0$ and such that for any $G$-line bundle $L$ on $U$, if $G$ acts trivially on the fiber $L_{u_0}$, then $\Stab_u$ acts trivially on the fiber $L_u$.
\end{prop}

\begin{proof} Following \cite[\S 4]{Drezet-Narasimhan}, we claim that up to passing to a saturated open neighborhood of $u_0 \in U$, we can assume that for every $u \in U$ with closed orbit there exist a closed irreducible subvariety $Z$ containing $u$ and $u_0$ and with the property that there is an open subset $Z^\circ \subset Z$, containing $u$ and where the polystable type is constant. Let us show why this holds. Let $\underline \alpha_0$ be the polystable type of $F$. Up to passing to a open neighborhood of $F \in M_v(\sigma)$, we can assume that all other strata in an open neighborhood contain $P_{\underline \alpha_0}$, and hence $F$, in their closure. Up to passing to a saturated open  neighborhood of $U$, we can assume that this holds on the image of the \'etale morphism $U \sslash G \to M_v(\sigma)$.
For every $\underline \alpha$, let $U_{\underline \alpha}$ the preimage of $P_{\underline \alpha}$ in $U$. Then $u_0$ is contained in the closure of every $U_{\underline \alpha}$. 

Now let $W_{\underline \alpha} \subset U_{\underline \alpha}$ be the union of those irreducible components whose closure does not contain $u_0$. Then $W_{\underline \alpha} \subsetneq U_{\underline \alpha}$, since $u_0 $ belongs to the closure of any $U_{\underline \alpha}$.
Let $U':=U \setminus \cup_{\underline \alpha} W_{\underline \alpha}$. Note that $u_0 \in U'$. Since the orbit of $u_0$ is closed we can pass to a neighborhood $U'' \subset U'$ of $u_0$, which is saturated in $U$.
Let $u \in U''$ be a point with closed orbit. By Theorem \ref{luna} (2) the morphism $U \to \mathfrak{M}_v(\sigma)$ sends points with closed orbits to closed points of the stack. Since by Theorem \ref{thm gms}(1) these correspond precisely to polystable objects, $u \in U_{\underline \alpha}$ for some $\alpha$. Since by the choice of $U''$, $u \notin W_{\underline \alpha}$, the component of $U''_{\underline \alpha}:=U'' \cap U_{\underline \alpha}$ containing $u$ has $u_0$ in its closure. We let $Z^\circ$ be this component and let $Z$ be its closure.  This proves the claim.

Finally, we need to prove that for any $G$ line bundle $L$ on $U$, the desired property holds. This follows from  Lemma \ref{lemGline} below.
\end{proof}

For the following Lemma we assume that we have restricted $U$ so that it satisfies the claim in proof of Proposition. The reader should compare this to \cite[Lem 2.3]{Drezet}

\begin{lemma}  \label{lemGline} Suppose $U$ is such that for every $u \in U$ with closed orbit there exist a closed irreducible subvariety $Z$ containing $u$ and $u_0$ and with the property that there is an open subset $Z^\circ \subset Z$, containing $u$ and where the polystable type is constant. 
Let $L$ be a $G$-line bunde on $U$ and let $\ua$ be a polystable type and let $u \in U_{\ua}$ be a point. Then
\begin{enumerate}
\item[(i)] $G_u$ acts trivially on $L_u$  if and only if $G_v$ acts trivially on $L_v$ for all $v \in U_{\ua}$;
\item[(ii)] If $G$ acts trivially on $L_{u_0}$, then $G_v$ acts trivially on $L_v$ for all $v \in U_{\ua}$.
\end{enumerate}
\end{lemma}
\begin{proof}
(i) Let $H_{\ua} \subset G \times U_{\ua}$ be the stabilizer group scheme of the locally closed subset $U_{\ua}$. 
Let $\mc F \in D^b_{U-perf}(U \times X)$ be the family of objects corresponding to the morphism $U \to \mathfrak M_v$ and let $\mc F_{\ua}$ be its (derived) restriction to $U_{\ua}$. By Theorem \ref{propalper} (2), there is an isomorphism of $U_{\underline \alpha}$-group schemes $H_{\ua}/U_{\ua} \cong \Aut(\mc F_{\ua}/U_{\ua})$. 
A $G$-linearization of a line bundle $L$ on $U_{\ua}$ induces a morphism $\varphi \in \Hom_{U_\alpha-gp}(H_{\ua}, \mathbb G_{m, U_{\ua}})$ of group schemes over $U_{\ua}$. We wish to show that $\varphi$ is trivial at one point if and only if it is trivial at all points. 
Up to passing a base change $V_{\ua} \to U_{\ua}$, we can assume that the group scheme is trivial,
so the pullback of $\varphi$ to $V_{\ua}$ is locally trivial. Since $U_{\ua}$ is connected, the conclusion follows.

 
(ii) Consider a morphism $f: (D,d_0) \to (U,u_0)$ from a smooth pointed curve and such that $f(D\setminus d_0) \subset U_{\underline \alpha}$. Let $\mc F_D$ be the object in $D^b_{D-perf}(D \times X)$ induced by the morphism $D \to U_{\underline \alpha} \to \mathfrak M_v$. Then for $d \neq d_0$, $F_{d}$ is polystable of type $\alpha$, so $\Aut(\mc F_{d})\cong\prod GL(\alpha_i)$, while $\Aut(\mc F_{D,d_0})\cong GL(m)=G$. 
Up to passing to a finite cover we can assume that $\mc F_D =\oplus \mc F_i^{\oplus \alpha_i}$ where for each $i$, $\mc F_i$ is a family of semistable objects with Mukai vector $v_i:=m_i v_0$ and such that $\mc F_{i,d}$ is stable for $d \neq d_0$ and $\mc F_{i,d_0}=F_0^{m_i}$.
Consider the relative automorphism group $H \to D$ of $\mc F_D$. Up to restricting $D$, by base change and our choice of $\mc F_D$, we can assume that $H_{|D \setminus d_0}$ is the trivial group scheme. Clearly $H_{d_0}=GL(m)=G$, so as soon as $\alpha \neq (m,1)$,  $H \to D$ is not  flat; the  flat limit $H' \subset H$ of $H_{D\setminus d_0}$ in $H$ is also trivial and determines a closed embedding $\prod GL(\alpha_i) \subset GL(\sum \alpha_im_i)=GL(m)=G$.
Let $L$ be a $G$ line bundle on $U$ and consider its pullback to $D$. If $G$ acts trivially on $L_{u_0}$, then $H'_{d_0}$ acts trivially on $L_{u_0}$ and hence $H_d$ acts trivially on $L_d$ for all $d$. Using the first part of the Lemma the conclusion follows. 
\end{proof}

\begin{lemma} \label{quotfact} Let $G$ be a connected linear algebraic group, let $V$ be a factorial connected affine variety with an action of $ G$ and let $Z:=V \sslash G$. 
If for every $G$-line bundle $L$ and for every $v \in V$ with closed orbit the action of  the stabilizer $\Stab_v \subset G$ on $L_v$ is trivial, then $Z$ is factorial.
\end{lemma}

\begin{rem} \label{rempicg} Under these assumptions every $G$-line bundle on a $G$-invariant open subset $U \subset V$ can be lifted to a  $G$-line bundle on $V$. 
Indeed,  the factoriality of $V$ implies that the restriction morphism  $\Pic(V) \to \Pic(U)$ is surjective, while the connectedness of $G$ implies by \cite[Lem 2.2 and Prop 2.3]{Knop-Kraft-Vust} that the kernel of the forgetful maps $f:\Pic^{G}(V) \to \Pic(V)$ and  $g:\Pic^{G}(U) \to \Pic(U)$ are surjected upon by $\chi(G)$, the group of characters of $G$. By \cite[Lem. 2.2]{Knop-Kraft-Vust} the normality of $V$ and $U$ imply that the cokernel of $f$ and $g$ are both contained in $\Pic(G)$. Putting these things together one can  check that the restriction map  $\Pic^G(V) \to \Pic^G(U)$ is surjective.
As a side remark, note that if the codimension of the complement of $U$ in $V$ is greater or equal to $2$, then $\Pic^G(V) \to \Pic^G(U)$ is an isomorphism.
\end{rem}

\begin{proof} It is enough to show that every Weil divisor on $Z$ is Cartier. Let $\pi: V \to Z:=V \sslash G$ be the quotient morphism and let $V^s=\pi^{-1}( Z^{s}) $ be the inverse image of the smooth locus of $Z$.
Let $D \subset Z $ be a Weil divisor and let $L^s=\pi^* \mc O(-D_{|Z^s})$ be the pullback  to $V^s$ of the ideal sheaf of $D_{|Z^s}$. The line bundle $L^s$ on $V^s$ is $G$-linearized so, by Remark \ref{rempicg}, it lifts to a $G$-line bundle $L$ on $V$. By assumption, the stabilizers satisfy the descent conditions, so $L$ descends to a line bundle on $Z$ extending the ideal sheaf of $D_{|Z^s}$. By codimension reasons ($Z$ is normal), the divisor corresponding to this line bundle has to be equal to $D$.
\end{proof}

\begin{cor}\label{U/Gfac} (cf. \cite{Drezet, Kaledin-Lehn-Sorger})
Let $F=F_0^m$ a polystable object of type $(m,1)$,  let $U$ be the corresponding slice as in Theorem \ref{luna} and set $G=\Aut(F)=GL(m)$. Then $U \sslash G$ is factorial.
\end{cor}
\begin{proof}
By Proposition \ref{Ufac}, we can assume that $U$ is factorial. The corollary then follows from  Lemma \ref{quotfact} above applied to $U\sslash G=U \sslash PGL(m)$, and the fact that $PGL(m)$ has no non-trivial characters.
\end{proof}

The following theorem is the main result of this section.

\begin{thm} \label{locally fact} Let $v=mv_0$, with $v_0^2>0$, be not of OG10-type and let $\sigma \in \Stab^{\dagger}(\mc D)$ be $v$-generic. Then $M_v(\sigma)$ is factorial with terminal singularities.
\end{thm}
\begin{proof}

As in \cite{Kaledin-Lehn-Sorger}, we start by showing that $M_v(\sigma)$ is factorial in a neighborhood $M_v(\sigma)^\circ$ of the deepest stratum $P_{m,1}=M_{v_0}$ of the singular locus. Then we use Proposition \ref{corstabslice} to show that from the factoriality around the deepest stratum we can deduce the factoriality at all points of $M_v(\sigma)$.

 By Corollary \ref{U/Gfac}  the \'etale slices $U \sslash PGL(m)$ centered at the points of $P_{m,1}$ are  factorial. Hence there is an \'etale morphism from a factorial variety onto an open subset $M_v(\sigma)^\circ \subset M_v(\sigma)$ of $P_{m,1}$. This shows the factoriality of $M_v(\sigma)^\circ$.
Note that since $P_{(m,1)}$ is contained in the closure of every other stratum, see Remark \ref{deepeststratum}), then for every polystable type $\ua$, $P_{\ua} \cap M_v(\sigma)^\circ \neq \emptyset$ (by Theorem \ref{thmproj}, we can assume that $P_{\ua}$ is irreducible).

Now let $[F] \in M_v(\sigma)$ be a point corresponding to a $\sigma$-polystable object $\oplus F_i^{ \oplus \alpha_i}$ of type $\underline \alpha_0$ and set $G=\Aut(F)$. We will show that there exists an open neighborhood of $[F] \in M_v(\sigma)$ with the property that every Weil divisor passing through $[F]$ is Cartier. This implies that the ideal sheaf of every effective Weil divisor in $M_v(\sigma)$ is locally free at  $[F]$, and hence that the local ring of $M_v(\sigma)$ at $[F]$ is a UFD. This is enough to prove the theorem.

Let $(V,v_0)$ be an \'etale slice at $F$ with \'etale morphism $\epsilon: V \sslash G \to M_v(\sigma)$. Up to restricting $V$, we can assume that $V $ is factorial (Proposition \ref{Ufac}) and that it satisfies the conclusions of Proposition \ref{corstabslice}.
Let $W \subset M_v(\sigma)$ be the image of $V\sslash G$ in $M_v(\sigma)$ and set $W^\circ=W \cap M_v(\sigma)^\circ$. Similarly, denote by $V^\circ$ its preimage in $V$. Since $P_{\ua}$ is irreducible and $P_{\ua} \cap M_v(\sigma)^\circ \neq \emptyset$, for every polystable type $\ua$ for which $P_{\ua} \cap W \neq 0$, the open subset $P_{\ua} \cap W^\circ$ is also non empty.

 Now let $D $ be a Weil divisor passing through $[F]$. We  show that $D_{|W}$ is Cartier. This is equivalent to showing that the Weil divisor $\epsilon^{-1}(D)$ is Cartier on $V \sslash G$. 
Since $D_{| W ^\circ}$ is Cartier, so is  $\epsilon^{-1}(D_{| W ^\circ})$; let  $N^\circ$ be the corresponding line bundle  on $V^\circ \sslash G$. Set $L^\circ:=\pi^* N^\circ$. Then for every $v \in V^\circ$ with closed orbit, $G_v$ acts trivially on $L^\circ_v$.
By Remark \ref{rempicg}, the $G$-line bundle $L^\circ$ extends to a $G$-line bundle $L$ on $V$. Since for every polystable type $\ua$ appearing in $W$, $P_{\ua} \cap W^\circ \neq 0$, repeated applications of Lemma \ref{lemGline} imply that for every $v \in V$ with closed orbit, $G_v$ acts trivially on $L_v$. 
This proves the factoriality.

Finally, from Namikawa's terminality criterion for symplectic varieties (cf. Lemma \ref{lemterminal4} ) and Corollary  \ref{cordimsing} it follows that $M_v(\sigma)$ has terminal singularities.
\end{proof}

\section{Preliminaries on Irreducible symplectic varieties}

In this section we recall the main definitions  and first properties of irreducible symplectic varieties, then we prove some basic results needed in Section \ref{modspacesISV}. Many of the results that appear in this section have already appeared in the literature, or are well known to the experts; we include them in this section for the readers' convenience. We will focus the treatment on projective varieties, even though when dealing with deformations of projective symplectic varieties one naturally encounters non-projective one. We will try to highlight when this is the case, and refer the reader to   \cite{Bakker-Lehn-global} for the notion of symplectic K\"ahler spaces and their first properties.

\begin{defin}[Symplectic varieties \cite{Beauville-sing}]  \label{defin symplsing} A normal variety $M$ is called holomorphic symplectic if its smooth locus has a holomorphic symplectic form which extends to a holomorphic form on any resolution of the singularities of $M$.
\end{defin}

A symplectic variety has canonical and, in fact, rational Gorenstein singularities.
A holomorphic form on the smooth locus of a normal variety is called a reflexive form, so we will refer to the holomorphic symplectic form in the definition of symplectic variety as the reflexive symplectic form. 

\begin{defin}[Irreducible  symplectic variety \cite{Greb-Kebekus-Peternell-singular}] Let $M$ be a projective symplectic variety with (reflexive) symplectic form $\sigma$.
A projective symplectic variety $M$ is called an irreducible (holomorphic) symplectic variety if for every quasi-\'etale  cover $f: Y \to M$
the algebra of reflexive forms on $Y$ is generated by the (reflexive) pullback of $\sigma$.
\end{defin}

Recall that a quasi-\'etale cover is a finite surjective morphism of normal varieties that is \'etale in codimension one. 

\begin{example}
Recall that an irreducible holomorphic symplectic manifold  is a simply connected compact K\"ahler manifold whose space of holomorphic $2$-forms is spanned by a holomorphic symplectic form \cite{Beauville83}. If $M$ is a (projective) irreducible holomorphic symplectic manifold, then $M$ is an irreducible holomorphic symplectic variety. Indeed,  it is simply connected and  a holonomy computation \cite[Prop. 3]{Beauville83} shows that its algebras of holomorphic forms is generated by the holomorphic symplectic form.
 Conversely, if $M$ is a smooth irreducible holomorphic symplectic variety, then $M$ is an irreducible holomorphic symplectic manifold: indeed, by \cite[Thm 1]{Schwald-Definition} $M$ is irreducible holomorphic symplectic manifold if and only if it is holomorphic symplectic and has no irregularity.
\end{example}

The following example shows that being irreducible holomorphic symplectic variety is not necessarily preserved by birational modifications (cf. Remark \ref{remHilbnISV}).

\begin{example} \label{HilbnotISV}
Let $S$ be a projective K3 surface and let $h: S^{[n]} \to  S^{(n)}$ be the Hilbert-Chow morphism from the Hilbert scheme of $n \ge 2$ points on $S$ to the symmetric $n$-th product of $S$. Then $S^{[n]}$ is an irreducible symplectic variety but $ S^{(n)}$ is not, since it admits a quasi-\'etale cover from the Cartesian product $S^n$.
\end{example}

The relevance of irreducible holomorphic symplectic variety  is clear from the following theorem, which generalizes the Beauville-Bogomolov decomposition theorem \cite{Beauville83} and whose proof  is the result of a joint effort\cite{Druel-Guenancia-DT, Durel-DT, Greb-Guenancia-Kebekus, Horing-Peternell-foliations}.

\begin{thm}(Singular decomposition theorem,  \cite[Thm 1.5]{Horing-Peternell-foliations})
Let $Z$ be a normal projective variety with klt singularities and numerically trivial canonical class. Then there exists a quasi-\'etale finite cover $f: Y \to Z$ where $Y$ is a projective variety with canonical singularities which decomposes in the product of abelian varieties, irreducible symplectic varieties, and irreducible Calabi-Yau varieties. Such a decomposition of $Y$ is unique up to reordering of the factors.
\end{thm}

By definition (see \cite{Greb-Kebekus-Peternell-singular}), an irreducible Calabi-Yau variety is a normal projective variety  with canonical singularities and trivial canonical class  and whose algebra of reflexive forms, along with the algebra of reflexive forms on any quasi-\'etale finite cover, is generated by a nowhere vanishing reflexive form of top degree. Note that there is a singular decomposition theorem also in the K\"ahler  setting, see \cite{Bakker-Guenancia-Lehn}.

A special role is played by symplectic varieties that are $\Q$-factorial and terminal, since in many ways they behave like smooth ones. One of the most important results for this class of symplectic varieties is Namikawa's theorem on locally trivial deformations, which we recall below in Theorem \ref{thmnamikawa}. It is thus important to have characterizations of  terminal and of $\Q$-factorial symplectic varieties.

\begin{lemma}(\cite[Cor 1]{Namikawa-Note}) \label{lemterminal4} Let $X$ be a symplectic variety. Then $X$ has terminal singularities if and only if the singular locus of $X$ has codimension $\ge 4$.
\end{lemma}


In terms of $\Q$-factoriality, we will use Koll\'ar-Mori's characterization of the $\Q$-factoriality for projective rational varieties (see \cite[12.1.6]{Kollar-Mori-flips}, or also (\ref{KM}) below) (and tangentially we will also use the analogous criterion formulated by Bakker-Lehn for K\"ahler non projective ones \cite[\S 2.12]{Bakker-Lehn-global-moduli}).

Next, we turn to deformations and families. Note that when considering deformations of projective symplectic varieties, one naturally leaves the projective category.  For this reason, the total space and the parameter space of any family that we will consider below will be a complex space and the proper morphisms below will not be projective in general. We refer the reader to  \cite{Bakker-Lehn-global-moduli} for a more complete treatment on these topics.

\begin{defin} Let $f: \mc X \to T$ be a family of proper complex spaces, i.e. a proper flat morphism of complex spaces with connected fibers. If $0 \in T$ is a distinguished point and $T$ is connected, we call $\mc X$ a deformation of $\mc X_0$.
  The family (or deformation) $\mc X$ is said to be locally trivial at $0 \in T$, if for all $x \in \mc X_0$ the deformation $(\mc X, x) \to (T, 0)$ of germs is trivial, i.e. $(\mc X, x) \cong (\mc X_0, x) \times (T, 0)$.
The family $\mc X$ is said to be locally trivial  if it is  locally trivial at $t$ for all $t \in T$.\\
We say that two proper varieties $X_1$ and $X_2$ are deformation equivalent under a locally trivial deformation if there exists a locally trivial family $\mc X \to T$ over a connected bases and two points $t_1, t_2 \in T$ such that $\mc X_{t_1} \cong X_1$ and $\mc X_{t_2} \cong X_2$.
\end{defin}

The proof of the following Lemma was communicated to us by B. Bakker.

\begin{lemma} \label{lem loctrivZariski}
Let $f:\mc X \to T$ be a family of proper complex spaces over a smooth $T$. If $f$ is locally trivial at a point $t \in T$, then $f$ is locally trivial on a Zariski open subset of $T$. 
\end{lemma}
\begin{proof}
The locus where $f$ is locally trivial is the locus where the natural map $T_T \to \Ext^1(\Omega_{\mc X/T},\mc O_{\mc X})$ lands in $H^1(T_{\mc X/T})$. Since  at $t \in T$ this happens to every order, the same holds on a Zariski open neighborhood.
\end{proof}

As a consequence of Thom's first isotopy Lemma, two proper complex spaces that are deformation equivalent under a locally trivial deformation are homeomorphic. The following proposition is known to the experts (for example, see \cite[Prop. 5.1]{Amerik-Verbitsky} for a slightly different proof).

\begin{prop} \label{local triviality} \label{simres}
Let $f: \mc X \to T$ be a locally trivial family of reduced complex spaces. Suppose $T$ is smooth and connected. Then
\begin{enumerate}
\item $f$ is topologically a fiber bundle, i.e. for every $t \in T$ there is an open (in the classical topology) neighborhood $V \subset T$ and a homeomorphism $f^{-1}(V)  \stackrel{\sim}{\to} \mc X_t \times V$ compatible with the two natural morphisms to $V$. Moreover, this homeomorphism is compatible with the stratification of $f^{-1}(V)$ given by the smooth locus, the smooth locus of the singular locus etc.
\item There exists a simultaneous resolution of singularities, i.e. a proper bimeromorphic morphism $\pi: \wt {\mc X} \to \mc X$ over $T$ such that $\wt {\mc X} \to T$ is smooth and such that $\pi$ is an isomorphism over the smooth locus of $f$. Moreover, the resolution is locally trivial, i.e., it is obtained by successive blow ups at centers which are smooth over $T$ and which locally are products.
\end{enumerate}

\end{prop}
\begin{proof}
(1) The assumption on $f$ implies that there is a Whitney stratification \cite[\S 1.2]{Goresky-Macpherson-stratified} $\{{X_\alpha}\}_{\alpha \in A}$ of $X$ that is compatible with $f$ in the sense that for all $\alpha \in A$ the induced morphism $X_\alpha \to T$ is smooth. Indeed, since the family of locally trivial, the canonical Whitney stratification of Le-Teissier (see \cite[pg. 70 and Cor 1.3.3]{Le-Teissier}) is induced locally by the canonical stratification of a given fiber, so it is locally a product. Since the strata are locally a product, the induced morphism on every stratum is smooth. 
By \cite[(4.14)]{Verdier}, this implies that for all $t \in T$, there is an open neighborhood $V$ of $t \in T$ and a homeomorphism $f^{-1}(V)  \to \mc X_t \times V$ compatible with the natural morphisms to $V$ and with the stratifications. 
(2) This is \cite[Lem. 4.9]{Bakker-Lehn-global-moduli} and its proof.
\end{proof}

For a projective variety $Z$, denote by $\Def(Z)$  the versal deformation space of the deformation functor of $Z$ and by $\Def(Z)_{lt}$ the versal deformation space of the deformation functor parametrizing locally trivial deformations of $Z$. These exists as germ of a complex spaces (or as the spectra of complete local algebras) and there is a natural closed embedding $\Def(Z)_{lt} \subset \Def(Z)$. There is a versal deformation $\mc Z \to \Def(Z)$ whose restriction to $\Def(Z)_{lt}$ is a locally trivial deformation of $Z$. The following is the Main Theorem of  \cite{Namikawa-Qfact}.


\begin{thm}[Namikawa] \label{thmnamikawa}
Let $M$ be a  $\Q$-factorial terminal projective symplectic variety. Then the closed embedding
\[
\Def(M)_{lt} \subseteq \Def(M)
\]
is an isomorphism and  both spaces are smooth of the expected dimension.  The restriction of the versal family over $\Def(M)$ is a locally trivial deformation of $M$ (and the versal family for $\Def(M)_{lt}$).
\end{thm}

\begin{lemma} \label{lemsmalldef}
Let $M$ be a projective $\Q$-factorial terminal symplectic variety and let $\mc M \to T$ be a deformation of $M$, with $\mc M_0=M$ for some $0 \in T$ and $T$ smooth.  Then the locus where $\mc M_t$ is not $\Q$-factorial terminal symplectic is a Zariski closed subset of $T$.

\end{lemma}

Note that in this Lemma we are surreptitiously using the notion of $\Q$-factoriality in the non algebraic setting. Since ultimately we are only concerned with projective varieties we will not enter in discussions on this subtle notion and instead refer the reader to  \cite[\S 2.12]{Bakker-Lehn-global-moduli}.

\begin{proof} Up to doing a base change to a resolution of $T$, we can assume that $T$ is smooth.  By Lemma \ref{lem loctrivZariski}, up to passing to a Zariski open subset we can assume that $\mc M \to T$ is a locally trivial deformation of $M$.
Then by \cite[Lemma 4.9]{Bakker-Lehn-global-moduli} there exists a simultaneous resolution of $\mc M$, i.e. a proper bimeromorphic morphism $\pi: \mc Y \to \mc M$ which is an isomorphism over the smooth locus of $\mc M \to T$. 
%
Since the sheaf $(f \circ \pi)_* \Omega^2_{{\mc Y}/T}$ is locally free, up to further restricting $T$ to  Zariski open neighborhood of $0 \in T$, we can find a local section extending $\sigma_0$. This defines a section $\sigma \in \Gamma(\mc M^{reg},  \Omega^2_{\mc M}/T)$, and since $\sigma_0$ is non-degenerate, up to further restricting $T$, we can assume that $\sigma$ is non-degenerate on $\mc M^{reg}_t$, for every $t \in T$.

Since $M=\mc M_0$ is projective, $\Q$-factorial and with rational singularities, by \cite[(12.1.10)]{Kollar-Mori-flips}
\be \label{KM}
\im[H^2(\mc Y_0,\mathbb Q) \to H^0(M, R^2 \pi_* \Q_{\mc Y_0})]=\im[\sum \Q [E_i] \to H^0(M, R^2 \pi_* \Q_{\mc Y_0}) ]
\ee
holds. Here the $E_i$ are the components of the exceptional locus of the resolution $\mc Y_0 \to M_0$. Since $\mc Y \to \mc M$ is a locally trivial resolution,  (\ref{KM}) holds for any $t \in T$. In particular, restricting the left handside of (\ref{KM}) to $\Pic(\mc Y_t)$ gives that
\[
\im[\Pic(\mc Y_t)_\mathbb Q \to H^0(M, R^2 \pi_* \Q_{\mc Y_t})]=\im[\sum \Q [E_{i,t}] \to H^0(M, R^2 \pi_* \Q_{\mc Y_0})]
\]
holds for every $t \in T$. By  \cite[Prop. 2.15]{Bakker-Lehn-global-moduli}, $\mc M_t$ is $\Q$-factorial for every $t \in T$. (Note that since having rational singularities is an analytic condition and $\mc M \to T$ is locally trivial, $\mc M_t$ has rational singularities for every $t \in T$.)
\end{proof}

Next we show that $\Q$-factorial terminal symplectic projective varieties are deformation equivalent, generalizing the well known result of Huybrechts for  irreducible holomorphic symplectic manifolds (note that Huybrechts' result does not require the projectivity assumption). For a different proof of the following result, see \cite[Thm 6.16]{Bakker-Lehn-global-moduli}.
\begin{prop} \label{propbirdef}
Let $\phi: M \dashrightarrow M'$ be a birational map of $\Q$-factorial terminal symplectic projective varieties. Then $M$ and $M'$ are deformation equivalent under a locally trivial deformation. In particular, they are homeomorphic.
\end{prop}
\begin{proof}
By Kawamata \cite[Thm 1]{Kawamata-flops}, the birational map $\phi$ factors as a finite sequence of flops. Note that since a flop of a  $\Q$-factorial terminal symplectic variety is still a $\Q$-factorial terminal symplectic variety, it is enough to prove the result for a flop. Let $\psi: Y^+ \dashrightarrow Y^-$ be a flop of $\Q$-factorial terminal symplectic projective varieties and let $\pi^\pm: Y^\pm \to Y_0$ be the associated flopping contractions. By Theorem 1 of \cite{Namikawa-Qfact}, there are finite surjective morphisms $\rho^\pm: \Def(Y^\pm)_{lt} \to \Def(Y_0)$ which are compatible with the versal families and induce an isomorphism between a general (locally trivial) deformation of $Y^\pm$ and a general deformation of $Y_0$. The proposition follows.
\end{proof}

\begin{rem}
Note that the proof above gives slightly more, namely that there are two locally trivial deformations $\mc M, \mc M' \to T$ of $M$ and $M'$ over a pointed disk $(T,0)$ and a birational map over $T$ which is an isomorphism away from the central fibers and which induces $\phi$ over $0$. The case when $M$ (and thus $M'$) admits a symplectic resolution was already proved in \cite[Thm 1.1]{Lehn-Pacienza}.

\end{rem}

The following Lemma establishes that if two $\Q$-factorial terminal symplectic projective varieties are deformation equivalent, then they are deformation equivalent with a locally trivial deformation.

\begin{lemma} \label{deformationquali} Let $f: \mc M \to T$ be a proper flat morphism of complex spaces, with $T$ connected. Let $t_1, t_2  \in T$ be points with $\mc M_{t_1}$ and  $\mc M_{t_2}$ $\Q$-factorial terminal symplectic projective varieties.  Then $\mc M_{t_1}$ and  $\mc M_{t_2}$ are deformation equivalent under a locally trivial deformation. In particular, they are homeomorphic.
\end{lemma}

\begin{proof} By Lemma \ref{lemsmalldef}  there exists a non empty open Zariski subset $U \subset T$ parametrizing $\Q$-factorial terminal symplectic varieties. Denote by $Z$ the complement of $U$ in $T$. If $t_1$ and $t_2$ belong to the same connected component of $U$, the statement follows from from Proposition \ref{local triviality}. Otherwise, $t_1 $ and $t_2$ lie in different irreducible components $T_1$ and $T_2$ of $T$ which are disconnected by the removal of $Z$. Choose $z \in T_1 \cap T_2 \cap Z$ and let $M_z$ be the corresponding fiber.  For $i=1,2$, there are morphisms of germs $(T_i,z) \to \Def (M_z)$ whose image intersects the locus parametrizing $\Q$-factorial deformation of $M_z$.  Let $\wt M_z \to M_z$ be a $\Q$-factorial terminalization (this exists by \cite[Cor. 1.4.3]{BCHM}). By the already cited \cite[Thm 1]{Namikawa-Qfact}, the natural morphism $\Def(\wt M_z) \to \Def (M_z)$ is a finite covering, so up to passing to a finite cover $(\wt T_i, \wt z)  \to (T_i,z)$ of germs, we get morphisms  $(\wt T_i, \wt z)  \to \Def (\wt M_z)$.
\end{proof}


\begin{prop}  \label{invariance} Let $M$ be an irreducible symplectic variety that is $\Q$-factorial and terminal. Suppose the smooth locus of $M$ is simply connected. Then
\begin{enumerate}
\item Any  deformation of $M$ that is $\Q$-factorial, terminal and symplectic is an irreducible symplectic variety.
\item Any $\Q$-factorial terminal symplectic projective variety that is birational to $M$ is also an irreducible symplectic variety.
\end{enumerate}
\end{prop}

Statement (1) was already used in \cite{Perego-Rapagnetta-Irr}). See also \cite[Cor. 3.10]{Bakker-Guenancia-Lehn} for a more general statement.

\begin{proof} 
(1) By the first item of Proposition \ref{deformationquali} any  deformation $M'$ of $M$ has simply connected singular locus, so any quasi-\'etale cover of $M'$ is an isomorphism and we only need to check the algebra of holomorphic forms of $M'$ itself.
Since the algebra of holomorphic forms is a birational invariant for proper complex space with canonical singularities, it suffices to check that the algebra of holomorphic forms on a resolution of $M'$ is generated by the (reflexive) pullback of the holomorphic symplectic form. The statement now follows from the existence of simultaneous resolutions (Proposition \ref{simres} (2)), and the fact that the $h^{p,0}$ are locally constant in a flat family of K\"ahler manifolds.
(2) This follows Proposition \ref{propbirdef} and part (1).
\end{proof}

\section{Moduli spaces as Irreducible symplectic varieties} \label{modspacesISV}
Let $\mc D$ be the derived category of a (twisted) K3 or the Kuznestov component of a smooth cubic fourfold.
For the rest of this section, fix an $m \ge 2$ and a $k>0$, not both equal to $2$.
 Let $v=mv_0 \in \wt K(\mc D)$ be a Mukai vector with $v_0^2=2k$ and let  $\sigma$ be a $v$-generic stability condition in $\Stab^\dagger(\mc D)$.
The main result of this section is Theorem \ref{thmisv}, namely that the projective moduli space $M_v(\mc D, \sigma)$ is an irreducible symplectic variety and that its (locally trivial) deformation class is uniquely determined by $m$ and $k$. This extends the results of main result of \cite{Perego-Rapagnetta-Irr} where this is proved that all Gieseker moduli spaces with  given $m$ and $k$ are deformation equivalent under a locally trivial deformation. We prove this by showing that we can reduce to the case when $\mc D$ is the derived category of a K3 surface, then we give a streamlined account the classical argument of O'Grady and Yoshioka \cite{OGradyWt2,Yoshioka-moduli-abelian} (see also  \cite{Perego-Rapagnetta-Irr}), to show that  the (locally trivial) deformation class of moduli spaces of Bridgeland stable objects on K3 surfaces with class $v$ as above, depends only on $m$ and $k$. 

\begin{thm} \label{thmisv} Let $\mc D$ be as above. Let $v=mv_0 \in K_{num}(\mc D)$ be a non-primitive Mukai vector not of OG10-type and let $\sigma$ be a $v$-generic stability condition in $\Stab^\dagger (\mc D)$. The the moduli space  $M_v(\mc D, \sigma)$ is an irreducible symplectic variety. Moreover, if $\mc D'$ is another triangulated category which also is the derived category of  a (twisted) K3 or the Kuznestov component of a smooth cubic fourfold,  $v'=m'v_0' \in K_{num}(\mc D')$ is a non-primitive Mukai vector not of OG10-type and $\sigma' \in \Stab^\dagger (\mc D')$ is a $v'$-generic stability condition, then $M_v(\mc D, \sigma)$ is deformation equivalent to $M_{v'}(\mc D', \sigma')$ if and only if $m=m'$ and $v_0^2={v'_0}^2$. When this is the case, $M_v(\mc D, \sigma)$ and $M_{v'}(\mc D, \sigma)$ are deformation equivalent under a locally trivial deformation.
\end{thm}

\begin{rem}  \label{remHilbnISV}
If  $\sigma_0$ is not $v$-generic, Example \ref{HilbnotISV} shows that  $M_v(\sigma)$ is not necessarily an irreducible symplectic variety.
\end{rem}


\subsection{Cubic fourfold case}

Let $X$ be a smooth cubic fourfold, fix $m \ge 2$ and let $v_0 \in \wt H_{alg}(Ku(X), \Z)$ be primitive. As usual, we assume that $m$ and $k$  are such that $v=mv_0$ is not of OG10 type. Let $\sigma \in \Stab^\dagger(Ku(X))$ be $v$-generic. 

As in \S \ref{sect ample}, consider a deformation $\mc X \to T$ of $X$ over which $v$ stays algebraic. There is a countable dense set of $t \in T$ for which $Ku(\mc X_t)\cong \mc D(S_t)$ for some K3 surface $S_t$. Let $M$ be as in the proof of Theorem \ref{thmproj}.
Up to passing to a finite cover of $T$ and possibly also to an open subset of $T$, there is a family $\famsigma \in \Stab_M(Ku(\mc X))$ of $v$-generic stability conditions extending $\sigma \in \Stab^\dagger(Ku(X))$. By Theorems \ref{thmproj} and \ref{locally fact}, the resulting family of moduli spaces
\[
M_v(Ku(\mc X), \famsigma) \to T
\]
has the property that the fibers are projective symplectic varieties with factorial and terminal singularities. Moreover, by Theorem \ref{thmnamikawa}, the family is locally trivial. Thus, $M_v(Ku(X),\sigma)$ is deformation equivalent under a locally trivial deformation to a moduli spaces of objects on a K3 surface, semistable with respect to a generic stability condition.

\subsection{The case of twisted K3 surfaces}

Let $(S,H)$ be a polarized K3 surface, let $ \alpha$ be a Brauer class with a  $B$-field  $B \in H^2(S,\mathbb Q)$ lifting $\alpha$.   We denote by $\mc D(S, \alpha)$ the derived category of $\alpha$-twisted coherent sheaves on $S$ (cf. \cite[\S 1]{Huybrechts-Stellari}) and set $\mc D=\mc D(S, \alpha)$. Recall that $H^*_{alg}(S,B, \Z) := (\exp(B) H^*_{alg}(X,\mathbb Q) \cap H^*_{alg}(S, \Z)\subset H^*_{alg}(S, \Z)$ contains the Chern characters of objects in $\mc D(S, \alpha)$.

We now record some results of \cite[\S 3]{Meachan-Zhang}.
The following proposition guarantees that, up to passing to an isomorphic moduli space, any moduli spaces of twisted objects can  be deformed to a moduli space of untwisted objects. This involves deforming the K3 surface together with an appropriate choice of the $B$-field lifting the Brauer class, to an untwisted K3 surface in such a way that the Mukai vector stays algebraic along the deformation.

\begin{prop} \label{prop untwistedK3}
Let $v=m v_0$, where $v_0 \in H^*_{alg}(S,B, \Z)$ with $v_0^2=2k \ge 2$ and let $\sigma \in \Stab^\dagger(S, \alpha)$ be a $v$-generic stability condition. 
Then there exists a K3 surface $S'$,  a Mukai vector $w=mw_0 \in  H^*_{alg}(S',\Z)$ with $w_0^2=2k$, a $w$-generic stability condition $\sigma' \in \Stab^\dagger(S')$, and a smooth connected curve $T$ parametrizing a family of moduli spaces 
\[
\mc M \to T
\]
 and two points $t_1, t_2 \in T$ such that $\mc M_{t_1}\cong M_v(S, \alpha, \sigma)$ and $\mc M_{t_2}\cong M_w(S, \sigma')$.
In particular, $M_v(S, \alpha, \sigma)$ can be deformed to the untwisted moduli space via a locally trivial deformation. Up to passing to an \'etale cover of $T$ we can also assume that the stable locus of this family of moduli spaces admits a quasi-universal family.
\end{prop}
\begin{proof}
In  \cite{Bayer-Macri-Proj} it is proved that any moduli spaces is isomorphic to a moduli space of Gieseker semistable  (twisted) sheaves of positive rank  (see the proof of Theorem 1.3 (a), pg. 730 of loc. cit).  So up to replacing $(S, \alpha)$ with another (twisted) K3 surface that is derived equivalent to the original one, we can assume that $M_v(S, \alpha, \sigma)$ is a moduli space of Gieseker $H$-semi-stable  (twisted) sheaves of positive rank, for some ample class $H$ on $S$. 
We are thus in the position of applying the results of \S 3.1 of \cite{Meachan-Zhang}, where it is proved that these moduli spaces can be deformed to untwisted moduli spaces. For the readers' convenience, we give an outline of their proof.

 Let $B$ be a rational $B$-field lifting $\alpha$. By Lemma 3.3 of  \cite{Meachan-Zhang}, up to tensoring by a multiple of $H$ (which multiplies the Mukai vector by $\exp(H)$) we can assume that $v$ is deformable in the sense of Definition 3.2 of loc. cit. This means that for any  deformation $f: \mc S \to T$ of $S$ where $H$ stays algebraic, and for any choice of local section of $R^2 f_* \mathbb Q_{\mc S}$ extending $B$, the Mukai vector $v$ deforms as a $B$-twisted algebraic class. By Lemma 3.6 of \cite{Meachan-Zhang}, up to  changing the $B$-field lifting $\alpha$, we can assume that $B^2<0$ and that $Span_{\Q}\{H, B\} \cap H^\perp \cap H^2(S,\mathbb Z)$ does not contain any $(-2)$-class.  We are thus in the position of applying Proposition 3.7 of \cite{Meachan-Zhang}, which shows that there exists a family $f: (\mc S, \mc H) \to T'$ of polarized K3 surfaces such $(\mc S_{t_1}, \mc H_{t_1})=(S,H)$, a section $\mc B$ of $R^2 f_* \mathbb Q_{\mc S}$ with the property that $\mc B_{t_1}$ lifts the Brauer class $\alpha$, and a section $v$  of $R^2 f_* \mathbb Z_{\mc S}$ such that $v_t \in H_{alg}^*(\mc S_t, B_t, \mathbb Z)$. Moreover, by Proposition 3.9 of  \cite{Meachan-Zhang} we can assume that $\mc H_t$ is $v_t$-generic for every $t \in T'$. Finally, by Proposition 3.10 of \cite{Meachan-Zhang}, up to passing to an \'etale cover of $T$ we can assume there is a relative family of moduli spaces.
The last statement is proved in \cite[Prop. 3.10]{Meachan-Zhang}
\end{proof}

\subsection{The case of K3 surfaces} \label{sec OGYoshPR}

In this section we streamline the classical argument that goes back to O'Grady and Yoshioka (see also  \cite{Perego-Rapagnetta-Irr}) that consists in deforming to an elliptic K3 surface in order to connect any two moduli spaces with the same numerical invariants (i.e. $m$ and $k$).

\begin{thm} \label{thm defclassk3}
Let $S$ be a K3 surface, $v=m v_0 \in H^*_{alg}(S,\Z)$ a positive Mukai vector with $v_0^2=2k \ge 2$  and let $\sigma \in \Stab^\dagger(S)$ be a $v$-generic stability condition. The locally trivial deformation class of the moduli space $M_v(S, \sigma)$ is uniquely determined by the integers $m$ and $v_0^2$.
\end{thm}

The case of OG10 moduli spaces is treated in \cite{Meachan-Zhang}. We will use the action on $D^b(S)$ of certain derived equivalences.

\begin{rem} \label{remAutD} By \cite[Lem 8.2]{Bridgeland-Annals}, the group $\Aut(\mc D^b(S))$ acts on the stability manifold $\Stab(S)$. Let  $\Aut^0(\mc D^b(S))$ be the subgroup preserving the connected component $\Stab^\dagger(S)$. For any $\Phi \in \Aut^0(\mc D^b(S))$ and any Mukai vector $v$, there is an isomorphism $M_v(\sigma) \to M_{\Phi(v)}(\Phi(\sigma) )$. Moreover, since the action of $\Phi$ on $\Stab^\dagger(S)$ sends the $v$-walls to the $\Phi(v)$-walls, up to moving slightly $\sigma$ in its chamber, we can assume that  $\sigma$ is both $v$ and $\Phi(v)$-generic. By Theorem \ref{Projnongeneric} (3) there is a birational map
\[
M_v(\sigma) \dashrightarrow M_{\Phi(v)}(\sigma).
\]
Note that because of this observation, when considering birational maps induced by these autoequivalences, one does not need to worry about preserving  the genericity of the stability condition.
We will use two types of autoequivalences that belong to $\Aut^0(\mc D^b(S))$:
\begin{enumerate}
\item Tensorization by a line bundle $A$, which acts on Mukai vectors by $v\mapsto\exp(\ch(A)) v$.
\item Spherical twist with respect to the spherical object $\mc O_S$, which acts on Mukai vectors by the reflection in the $-2$-class $v(\mc O_S)=(1,0,1)$. 
\end{enumerate}
\end{rem}

\begin{proof}[Proof of Theorem \ref{thm defclassk3}]
We claim that up to 
\begin{itemize}
\item[(a)] deforming the K3 surface while keeping the Mukai vector algebraic 
\item[(b)] acting by elements of $\Aut^0(\mc D^b(S))$
\end{itemize}
we can connect any two Mukai vectors of the form $mv_0$, with $v_0^2=2k$. 
The claim proves one part of theorem, namely that all moduli spaces with given $m$ and $k$ are deformation equivalent. Indeed, by the existence of stability conditions for families of surfaces \cite[Thm 24.1]{BLMNPS}, deforming the K3 surface while keeping the Mukai vector algebraic results in a deformation of the corresponding moduli spaces; as in \cite[Prop. 30.8]{BLMNPS} we can ensure that the stability conditions are $v$-generic, so that by Theorem \ref{locally fact}, and Theorem \ref{thmnamikawa}, the resulting family is locally trivial. 
Moreover, by Remark \ref{remAutD}, the action of $\Phi \in  \Aut^0(\mc D^b(S))$ induces a birational map of moduli spaces, and by Proposition \ref{propbirdef} and Theorem \ref{Projnongeneric}, for generic stability conditions the corresponding birational moduli spaces are deformation equivalent under a locally trivial deformation.
Finally, one  concludes using Propositions \ref{local triviality} and \ref{invariance}, and the fact that by \cite[Prop. 3.2]{Perego-Rapagnetta-Irr} the fundamental group of the smooth locus of Gieseker moduli spaces is simply connected.

Before proving the claim, we show that for different pairs of $m$ and $k$ the deformation classes are distinct. Indeed, suppose $M_v(\mc D, \sigma)$ deforms to $M_{v'}(\mc D', \sigma')$, for some Mukai vectors $v=mv_0$ and $v'=m'v_0'$ with $v_0^2=2k$ and ${v'}_0^2=2k'$. Then $M_v(\mc D, \sigma)$ and $M_{v'}(\mc D', \sigma')$ can be deformed into each other with a locally trivial deformation and hence by Proposition \ref{local triviality} the deepest stratum of preserved under this deformation. As a consequence, $M_{v_0}(\mc D, \sigma)$ deforms to $M_{v'_0}(\mc D', \sigma')$, and the conclusion follows. 

Let us now prove the claim.
For a primitive Mukai vector $v_0$, we will write $v_0=(r,sL,t)$, where $r,s,t \in \Z$ and $L \in \Pic(S)$ is primitive.
We now show that using (a) and (b) (in fact, only tensorization by a line bundle and spherical twist by $\mc O_S$ will be enough), we can transform any Mukai vector into the Mukai vector $(0,L,1)$ on an elliptic surface $(S,F)$ with section $R$ and where $L=R+(k-1)F$. The following are the steps:
\begin{enumerate}

\item By Remark \ref{pmv}, up to switching $v_0$ with $-v_0$, we can assume that $r \ge 0$. 

\item  \label{rzerook} If $r=0$, since $v_0^2=s^2 L^2=2k \ge 2$, again up to switching $v_0$ with $-v_0$, we can assume that $L$ is in the positive cone of $S$ and $s>0$. By the irreducibility of the Noether-Lefschetz locus parametrizing K3 surfaces with a line bundle of degree $2k$, we can assume that $S$ has an elliptic fibration $F$ with a section $R$ and $L=\mc O_S(R+(k-1)F)$. Tensoring by $\mc O_S((-s+1)F)$, we can assume that $t=1$, i.e. that $v=(0, R+(k-1)F,1)$ and hence we are done.

\item \label{rsge0} If $r >0$, up to tensoring by a power of $L$, we can assume that $s >0$.

\item  \label{rank2} We claim that if $r,s>0$, then we can assume that $r=s r'$ with $gcd(s,t)=1$. Indeed, since $L$ is primitive, up to deformation, we can assume that $\rho(S) = 2$ and that $L$ is one of the generators of $\Pic(S)$. In particular, we may assume that there exists a primitive line bundle $A \in \Pic(S)$, such the lattice spanned by $L$ and $A$ is saturated in $\Pic(S)$. Let $c=gcd(r,s)$ and set $r=cr''$ and $s=cs''$. There exists an $a \in \Z$ such that $s''L+ar''A$ is primitive and in the positive cone of $S$. Tensoring by $aA$, the second entry of the Mukai vector becomes $c(s''L+ar''A)$. In other words, up to deforming $S$ and tensoring by a line bundle, we can assume   $s|r$.

\item Suppose that $r=s r'$ with $gcd(s,t)=1$. Applying a spherical twist by $\mc O_S$ (which switches $r$ and $t$ and multiplies $s$ by $-1$), we can assume that $gcd(r,s)=1$;   then applying step (\ref{rsge0}) again, we can further assume that $s>0$. In other words, we can assume that $v_0=(r,sL,st)$ with $s >0$ and $gcd(s,r)=1$.

\item As in Step \ref{rank2}, up to deformation we can assume that $\rho(S) \ge 2$ and that there is primitive line bundle $A \in \Pic(S)$, such the lattice spanned by $L$ and $A$ is saturated in $\Pic(S)$. Up to tensoring by a multiple of $A$ we can assume that the second entry of $v_0$ is primitive and in the positive cone. Hence, we can assume that $s=1$ with $L$ in the positive cone.

\item As above, by the irreducibility of the Noether-Lefschetz locus parametrizing K3 surfaces with a line bundle of degree $2k$, we can assume that $S$ has an elliptic fibration $F$ with a section  $R$ and that $L=\mc O_S(R+(l-1)F)$. Here $l$ is such that $L^2=2l-2$. Since $v_0^2=2k=2l-2-2rt$, tensoring by $\mc O_S(tF)$, we can assume that $l-1=k$ and $t=0$. We use a spherical twist by $\mc O_S$, to assume that $r=0$.

\item Going back to Step (\ref{rzerook}) concludes the process.
\end{enumerate}

\end{proof}

\section{Second integral cohomology of moduli spaces} \label{sect second cohomology}

The aim of this section is to prove the following Theorem, which extends the well known results that hold for a) primitive Mukai vectors and Bridgeland stabilities on K3 surfaces and Kuznetsov components \cite{Mukai-Tata,OGradyWt2,Yoshioka-moduli-abelian,BLMNPS} b)  Giesker moduli spaces on K3 surfaces with  non-primitive Mukai vector that are not of OG10-type \cite{Perego-Rapagnetta-Irr}, c)  singular OG10-type moduli spaces  on K3 surfaces \cite{Meachan-Zhang}.

In this section we let $\mc D$ be the derived category of a  (twisted)  K3 surface or the Kuznetsov component of a smooth cubic fourfold and $X$ will denote either the K3 surface or the cubic fourfold.

\begin{thm} \label{vperp} 
Fix $m \ge 2$ and let $v_0 \in K_{num}(\mc D)$ be primitive. We assume that $v=mv_0$ is not of OG10 type. For any $\sigma \in \Stab^\dagger(\mc D)$ $v$-generic, let $M_v(\mc D, \sigma)$ be the moduli space of $\sigma$-semistable objects and let $M_v^s(\mc D, \sigma)$ be the stable locus, which coincides with the smooth locus. 
The Mukai homomorphism defined using a quasi-universal family on the stable locus induces an isomorphism
\[ 
v^\perp \cong H^2(M_v(\mc D, \sigma), \Z).
\]
\end{thm}

Here $v^\perp$ is a sublattice of the extended Mukai lattice of $\mc D$, i.e. of $H^*(S, \mathbb Z)$ in the case $\mc D$ is the derived category of a (twisted) K3 surface and $\wt H(\Ku(X),\Z)$ in the case $\mc D$ is the Kuznetsov component of a cubic fourfold.

\begin{proof} 
The morphism $v^\perp \to H^2(M_v(\mc D, \sigma), \Z) $ is defined as the composition of the morphism
\[
j^*: H^2(M_v(\mc D, \sigma), \Z)  \to H^2(M^s_v(\mc D, \sigma), \Z).
\]
induced by the open immersion $j: M_v^s(\mc D, \sigma) \subset M_v(\mc D, \sigma)$ of the smooth (i.e. stable) locus and the morphism
\be \label{vperpform}
\theta_v: v^\perp  \to H^2(M_v^s(\mc D, \sigma), \Z).
\ee
induced by a quasi-universal family  on $M_v^s(\mc D, \sigma)$ (see Definition A.4 of \cite{Mukai-Tata}). The definition of this morphism will be recalled below.

Following a path already set in \cite{Mukai-Tata} (see also \cite{OGradyWt2,Yoshioka-moduli-abelian,Perego-Rapagnetta-second-cohomology, Meachan-Zhang}), we claim that for both $j^*: H^2(M_v(\mc D, \sigma), \Z)  \to H^2(M^s_v(\mc D, \sigma), \Z)$ and $\theta_v: v^\perp  \to H^2(M_v^s(\mc D, \sigma), \Z)$, whether or not they are isomorphism is preserved by locally trivial deformations of moduli spaces (with quasi-universal families) and by birational morphisms induced by Fourier-Mukai transforms. 
Assuming we have proved this, the theorem follows by deforming to the locus of untwisted K3 surfaces (Proposition \ref{prop untwistedK3} and Lemma \ref{lem universal fam} below),  using the steps outlined in \S \ref{sec OGYoshPR}, and finally using the fact that these two claims are true for Gieseker semistable sheaves on a K3 surface, as showed in Proposition 3.5 (2) and Theorem 1.6 of   \cite{Perego-Rapagnetta-second-cohomology}.

We start by proving the claim for $j^*$. If $\mc M \to T$  is a locally trivial family parametrized by a smooth connected complex space $T$, then for any two points $t_1, t_2 \in T$, the restriction $H^2(\mc M_{t_1}, \Z) \to H^2(\mc M^{reg}_{t_1})$ is an isomorphism if and only if $H^2(\mc M_{t_2}, \Z) \to H^2(\mc M^{reg}_{t_2})$ is an isomorphism. This is clear, since  by Proposition \ref{local triviality} any locally trivial morphism $\mc M \to T$ with $T$ smooth induces a homeomorphism between any two fibers which is compatible with the inclusion of the strata of a Whitney stratification and, in particular, with the open embedding of their respective smooth loci. 
Now let $f: M_1 \dashrightarrow M_2$ be a birational morphism. By Proposition \ref{propbirdef} $M_1$ and $M_2$ are deformation equivalent under a locally trivial deformation, so the open embedding of smooth locus in $M_1$ induces an isomorphism in cohomology if and only if the same is true for the smooth locus of $M_2$.

We now construct the morphism from $v^\perp$ to $H^2(M_v^s(\mc D, \sigma), \mathbb Z)$, and prove that it behaves well in families and under birational morphisms induced by Fourier-Mukai transforms. 

Let $\mc E$ on $M_v^s(\mc D, \sigma) \times X$ be a quasi-universal family and let $\rho$ be its similitude, i.e. the  positive integer such that for one point $[F] \in M_v^s(\mc D, \sigma)$ (and hence for any point), $\mc E_{[F]} \cong F^{\oplus \rho}$ (cf. \cite[Def. A.4]{Mukai-Tata}). This exists by Theorem \ref{thm gms}. We define the following classical homomorphism  (see also \S 4 \cite{Perego-Rapagnetta-second-cohomology} and in particular Def. 4.2 of loc. cit.) by setting 
\begin{equation} \label{vperpformstable}
\begin{aligned}
\lambda_{\mc E}: v^\perp & \longrightarrow H^2(M^s_v(\mc D, \sigma), \mathbb Z)\\
 x & \longmapsto c_1 ({p_1}_*(p_2^*(x^\vee) \otimes_{top} [ \mc E])).
 \end{aligned}
\end{equation}
Here $[ \mc E]$ denotes the class in $K^0_{top}(M_v^s(\mc D, \sigma) \times X)$,  $\cdot \otimes_{top} \cdot$ (resp. $( \cdot)^\vee$) denotes the tensorization (resp. the dual) in topological K-theory, $c_1(\cdot)$ denotes the first Chern class, and for a proper morphism $f: Z \to Y$ of complex quasi-projective algebraic varieties $f_*$ denotes the pushforward under a proper morphism in topological K-theory. The reader can verify that the properties of this $\lambda_{\mc E}$ discussed in \cite[Lem 4.3]{Perego-Rapagnetta-second-cohomology} hold true also in the present setting.

The compatibility of this homomorphism with  birational morphisms induced by Fourier-Mukai transforms is checked as in \cite[Prop. 2.4 ]{Yoshioka-moduli-abelian}.

Before defining the (\ref{vperpform}) we need the following Lemma.

\begin{lemma} \label{lem universal fam}
Let $\mc X \to T$ be a family of cubic $4$ folds. Suppose that there is a section $v$ of $\wt H (\Ku(\mc X / T))$ which stays algebraic for every $t \in T$ and a family $\sigma_T=\{\sigma_t\}_{t \in T}$ of stability conditions as in \cite[Def. 21.15]{BLMNPS}. Let $f: \mc M_v(\mc D_T, \famsigma) \to T $ be the relative moduli space as in \cite[Thm. 21.24]{BLMNPS} and let $f^\circ: \mc M^s_v(\mc D_T, \famsigma) \to T $ be the  restriction to the stable locus. 
\begin{enumerate}
\item Then there exists a quasi-universal family $\mc E_T$ on $ \mc M^s_v(\mc D_T, \famsigma)$.
\item The family $\mc E_T$ defines a relative version of (\ref{vperpformstable}), i.e. a morphism of local systems
\[
\begin{aligned}
\lambda_{\mc E_T}: v^\perp & \longrightarrow R^2_{f^\circ_*}(M^s_v(\mc {K}u(X), \sigma), \mathbb Z)\\
 x & \longmapsto c_1 ({p_{T,1}}_*(p_{T,2}^*(x^\vee) \otimes_{top} [ \mc E])).
\end{aligned}
\]
\end{enumerate}
\end{lemma}
\begin{proof}
(1) The same proof of part (3) of Theorem \ref{thm gms} goes through, mutatis mutandis.
\end{proof}

Similarly, in the setting of Proposition \ref{prop untwistedK3}, the quasi-universal family on the relative moduli space defines a relative version of (\ref{vperpformstable}) as in (2) of the Lemma above.

The relative version allows us to define (\ref{vperpform}) by deforming to the locus where we know what happens.
\begin{lemma}
Let the notation be as above. The morphism $\lambda_{\mc E}$ defined in (\ref{vperpformstable}) is divisible by $\rho$, i.e., the morphism \begin{equation} \label{mudefinition}
\begin{aligned}
\theta_{\mc E}: v^\perp & \longrightarrow H^2(M^s_v(\mc D, \sigma), \mathbb Z)\\
 x & \mapsto \frac{1}{\rho}\lambda_{\mc E}(x) 
 \end{aligned}
\end{equation}
is well-defined. Moreover, $\mu_{\mc E}$ is independent of the quasi-universal family and compatible with birational morphisms of moduli spaces induced by Fourier-Mukai transforms. 
\end{lemma}
Following the standard notation, we call this morphism $\theta_v.$
\begin{proof}
We need to show that for every $ x\in v^\perp$, $\frac{1}{\rho}\lambda_{\mc E}(x) \in H^2(M^s_v(\mc D, \sigma), \mathbb Z)$. This implies the second statement, since by \cite[Lem. 4.3 (3) ]{Perego-Rapagnetta-second-cohomology}, if $x \in v^\perp$, then for every vector bundle $V$ of rank $r$ on $M^s_v(\mc D)$, we have $\lambda_{\mc E\otimes p_1^* V}(x)=r \lambda_{\mc E} $. Hence the desired independence follows from the properties of quasi-universal families.
To prove that $\frac{1}{\rho}\lambda_{\mc E}$ is well defined, by Lemma \ref{lem universal fam} and Proposition \ref{prop untwistedK3} it suffices to show that $\mc D$ can be deformed to one where this property holds. By the Steps in \S  \ref{sec OGYoshPR}, this follows from Prop. 4.4 (2) of \cite{Perego-Rapagnetta-second-cohomology}, which asserts that the result is true for Gieseker moduli spaces of sheaves on K3 surfaces, and by the compatibility of the morphism $\lambda$ with Fourier-Mukai transforms (see \cite[Prop. 2.4]{Yoshioka-moduli-abelian}).
\end{proof}

The same argument that shows that $\theta_v$ is well-defined shows that $\theta_v$ is an isomorphism, because up to deformations and birational transformations we can reduce the statement to Gieseker moduli spaces where it is know by \cite{Perego-Rapagnetta-second-cohomology}.
\end{proof}

\section{The case of even dimensional Gushel-Mukai manifolds} \label{sect GM}

In this last section we mention the fact  that, by using instead of \cite{BLMNPS} the corresponding results from \cite{Perry-Pertusi-Zhao},  all the main results of this paper ( i.e. Theorems \ref{thmproj}, \ref{locally fact}, \ref{thmisv} and \ref{vperp}) hold for moduli spaces of objects in the Kuznetsov component of four and six dimensional Gushel-Mukai manifolds, provided the class is not of OG$10$ type and the stability condition $v$-generic. In fact, even more can be proved. Namely that these theorems hold  for any K3 category $\mc D$ with the following properties. First of all, stability conditions have to exists and so do good moduli spaces of Bridgeland semistable objects. Second, we need that $\mc D$ can be deformed to a  category which is equivalent  to the derived category of a (twisted) K3 surface in such a way that family of stability conditions and the corresponding relative moduli spaces exist, and that the stability conditions obtained by specialization of these families of stability conditions to the K3 locus belong to the main connected component $\Stab^\dagger$. (In the proof above we have used the existence of families of cubic fourfolds with the property that there are infinitely many points where the category is equivalent to the derived category of a (twisted) K3 surface; however,  with a little extra care of what happens for non $v$-generic polarizations, it is enough to ask that one point has this property rather than infinitely many.)


\bibliographystyle{alpha}

\bibliography{ModuliKuznetsovBiblio.bib}

\newcommand{\etalchar}[1]{$^{#1}$}
\begin{thebibliography}{BCHM10}

\bibitem[ACV03]{Abramovich-Corti-Vistoli}
Dan Abramovich, Alessio Corti, and Angelo Vistoli.
\newblock Twisted bundles and admissible covers.
\newblock volume~31, pages 3547--3618. 2003.
\newblock Special issue in honor of Steven L. Kleiman.

\bibitem[AH61]{Atiyah-Hirzebruch}
M.~F. Atiyah and F.~Hirzebruch.
\newblock Vector bundles and homogeneous spaces.
\newblock In {\em Proc. {S}ympos. {P}ure {M}ath., {V}ol. {III}}, pages 7--38.
  American Mathematical Society, Providence, R.I., 1961.

\bibitem[AHLH18]{Alper-HalpernLeistner-Heinloth}
Jarod Alper, Daniel Halpern-Leistner, and Jochen Heinloth.
\newblock Existence of moduli spaces for algebraic stacks.
\newblock https://doi.org/10.48550/arxiv.1812.01128, 2018.

\bibitem[AHR20]{AHR}
Jarod Alper, Jack Hall, and David Rydh.
\newblock A {L}una \'{e}tale slice theorem for algebraic stacks.
\newblock {\em Ann. of Math. (2)}, 191(3):675--738, 2020.

\bibitem[Alp13]{Alper-gms}
J.~Alper.
\newblock Good moduli spaces for {A}rtin stacks.
\newblock {\em Ann. Inst. Fourier (Grenoble)}, 63(6):2349--2402, 2013.

\bibitem[Alp23]{Alper-Notes-Stacks}
J.~Alper.
\newblock Stacks and moduli.
\newblock \url{https://sites.math.washington.edu/~jarod/moduli.pdf}, 2023.
\newblock version of February 23, 2023.

\bibitem[Anc87]{Ancona}
Vincenzo Ancona.
\newblock Vanishing and nonvanishing theorems for numerically effective line
  bundles on complex spaces.
\newblock {\em Ann. Mat. Pura Appl. (4)}, 149:153--164, 1987.

\bibitem[AS18]{AS}
E.~Arbarello and G.~Sacc\`a.
\newblock Singularities of moduli spaces of sheaves on {K}3 surfaces and
  {N}akajima quiver varieties.
\newblock {\em Adv. Math.}, 329:649--703, 2018.

\bibitem[AS23]{ASII}
E.~Arbarello and G.~Sacc\`a.
\newblock Singularities of moduli spaces of objects in k3 categories: an
  update.
\newblock 2023.

\bibitem[AT14]{Addington-Thomas}
N.~Addington and R.~Thomas.
\newblock Hodge theory and derived categories of cubic fourfolds.
\newblock {\em Duke Math. J.}, 163(10):1885--1927, 2014.

\bibitem[AV21]{Amerik-Verbitsky}
Ekaterina Amerik and Misha Verbitsky.
\newblock Contraction centers in families of hyperk\"{a}hler manifolds.
\newblock {\em Selecta Math. (N.S.)}, 27(4):Paper No. 60, 26, 2021.

\bibitem[Bay19]{Bayer}
Arend Bayer.
\newblock A short proof of the deformation property of {B}ridgeland stability
  conditions.
\newblock {\em Math. Ann.}, 375(3-4):1597--1613, 2019.

\bibitem[BBD82]{BBD}
A.~A. Be\u{\i}linson, J.~Bernstein, and P.~Deligne.
\newblock Faisceaux pervers.
\newblock In {\em Analysis and topology on singular spaces, {I} ({L}uminy,
  1981)}, volume 100 of {\em Ast\'{e}risque}, pages 5--171. Soc. Math. France,
  Paris, 1982.

\bibitem[BCHM10]{BCHM}
Caucher Birkar, Paolo Cascini, Christopher~D. Hacon, and James McKernan.
\newblock Existence of minimal models for varieties of log general type.
\newblock {\em J. Amer. Math. Soc.}, 23(2):405--468, 2010.

\bibitem[BD85]{Beauville-Donagi}
A.~Beauville and R.~Donagi.
\newblock {La vari{\'{e}}t{\'{e}} des droites d{'}une hypersurface cubique de
  dimension $4$}.
\newblock {\em {C. R. Acad. Sc. Paris}}, {301}:703--706, {1985}.

\bibitem[Bea83]{Beauville83}
A.~Beauville.
\newblock {Vari{\'{e}}t{\'{e}}s {K{\"{a}}hleriennes} dont la premi{\`{e}}re
  classe de {Chern} est nulle}.
\newblock {\em {J. Diff. Geom.}}, {18}:755--782, {1983}.

\bibitem[Bea00]{Beauville-sing}
Arnaud Beauville.
\newblock Symplectic singularities.
\newblock {\em Invent. Math.}, 139(3):541--549, 2000.

\bibitem[BGL22]{Bakker-Guenancia-Lehn}
Benjamin Bakker, Henri Guenancia, and Christian Lehn.
\newblock Algebraic approximation and the decomposition theorem for k{\"a}hler
  calabi--yau varieties.
\newblock {\em Inventiones mathematicae}, 228(3):1255--1308, 2022.

\bibitem[BL21]{Bakker-Lehn-global}
Benjamin Bakker and Christian Lehn.
\newblock A global {T}orelli theorem for singular symplectic varieties.
\newblock {\em J. Eur. Math. Soc. (JEMS)}, 23(3):949--994, 2021.

\bibitem[BL22]{Bakker-Lehn-global-moduli}
Benjamin Bakker and Christian Lehn.
\newblock The global moduli theory of symplectic varieties.
\newblock {\em J. Reine Angew. Math.}, 790:223--265, 2022.

\bibitem[BLM{\etalchar{+}}19]{BLMNPS}
Arend Bayer, Mart{\'\i} Lahoz, Emanuele Macr{\`\i}, Howard Nuer, Alexander
  Perry, and Paolo Stellari.
\newblock Stability conditions in families, 2019.

\bibitem[BLMS17]{BLMS}
A.~Bayer, M.~Lahoz, E.~Macr\`i, and P.~Stellari.
\newblock {Stability conditions on Kuznetsov components}.
\newblock {Available at \url{https://arxiv.org/abs/1703.10839}}, {2017}.

\bibitem[BM14a]{Bayer-Macri-MMP}
A.~Bayer and E.~Macr{\`{\i}}.
\newblock M{MP} for moduli of sheaves on {K}3s via wall-crossing: nef and
  movable cones, {L}agrangian fibrations.
\newblock {\em Invent. Math.}, 198(3):505--590, 2014.

\bibitem[BM14b]{Bayer-Macri-Proj}
A.~Bayer and E.~Macr{\`{\i}}.
\newblock Projectivity and birational geometry of {B}ridgeland moduli spaces.
\newblock {\em J. Amer. Math. Soc.}, 27(3):707--752, 2014.

\bibitem[BMM21]{Bandiera-Manetti-Meazzini}
Ruggero Bandiera, Marco Manetti, and Francesco Meazzini.
\newblock Formality conjecture for minimal surfaces of {K}odaira dimension 0.
\newblock {\em Compos. Math.}, 157(2):215--235, 2021.

\bibitem[Bri07]{Bridgeland-Annals}
Tom Bridgeland.
\newblock Stability conditions on triangulated categories.
\newblock {\em Ann. of Math. (2)}, 166(2):317--345, 2007.

\bibitem[Bri08]{Bridgeland-K3}
T.~Bridgeland.
\newblock Stability conditions on {$K3$} surfaces.
\newblock {\em Duke Math. J.}, 141(2):241--291, 2008.

\bibitem[BZ19]{Budur-Zhang}
Nero Budur and Ziyu Zhang.
\newblock Formality conjecture for {K}3 surfaces.
\newblock {\em Compos. Math.}, 155(5):902--911, 2019.

\bibitem[DG18]{Druel-Guenancia-DT}
St\'{e}phane Druel and Henri Guenancia.
\newblock A decomposition theorem for smoothable varieties with trivial
  canonical class.
\newblock {\em J. \'{E}c. polytech. Math.}, 5:117--147, 2018.

\bibitem[DN89]{Drezet-Narasimhan}
J.-M. Drezet and M.~S. Narasimhan.
\newblock Groupe de {P}icard des vari\'{e}t\'{e}s de modules de fibr\'{e}s
  semi-stables sur les courbes alg\'{e}briques.
\newblock {\em Invent. Math.}, 97(1):53--94, 1989.

\bibitem[Dre91]{Drezet}
J.-M. Drezet.
\newblock Points non factoriels des vari\'{e}t\'{e}s de modules de faisceaux
  semi-stables sur une surface rationnelle.
\newblock {\em J. Reine Angew. Math.}, 413:99--126, 1991.

\bibitem[Dru18]{Durel-DT}
St\'{e}phane Druel.
\newblock A decomposition theorem for singular spaces with trivial canonical
  class of dimension at most five.
\newblock {\em Invent. Math.}, 211(1):245--296, 2018.

\bibitem[Fle88]{Flenner}
Hubert Flenner.
\newblock Extendability of differential forms on nonisolated singularities.
\newblock {\em Invent. Math.}, 94(2):317--326, 1988.

\bibitem[Fuj83]{Fujiki}
Akira Fujiki.
\newblock On primitively symplectic compact {K}\"{a}hler {$V$}-manifolds of
  dimension four.
\newblock In {\em Classification of algebraic and analytic manifolds ({K}atata,
  1982)}, volume~39 of {\em Progr. Math.}, pages 71--250. Birkh\"{a}user
  Boston, Boston, MA, 1983.

\bibitem[GGK19]{Greb-Guenancia-Kebekus}
Daniel Greb, Henri Guenancia, and Stefan Kebekus.
\newblock Klt varieties with trivial canonical class: holonomy, differential
  forms, and fundamental groups.
\newblock {\em Geom. Topol.}, 23(4):2051--2124, 2019.

\bibitem[GKP16]{Greb-Kebekus-Peternell-singular}
Daniel Greb, Stefan Kebekus, and Thomas Peternell.
\newblock Singular spaces with trivial canonical class.
\newblock In {\em Minimal models and extremal rays ({K}yoto, 2011)}, volume~70
  of {\em Adv. Stud. Pure Math.}, pages 67--113. Math. Soc. Japan, [Tokyo],
  2016.

\bibitem[GM88]{Goresky-Macpherson-stratified}
Mark Goresky and Robert MacPherson.
\newblock {\em Stratified {M}orse theory}, volume~14 of {\em Ergebnisse der
  Mathematik und ihrer Grenzgebiete (3) [Results in Mathematics and Related
  Areas (3)]}.
\newblock Springer-Verlag, Berlin, 1988.

\bibitem[HMS08]{Huybrechts-Stellari-Stability}
Daniel Huybrechts, Emanuele Macr\`\i, and Paolo Stellari.
\newblock Stability conditions for generic {$K3$} categories.
\newblock {\em Compos. Math.}, 144(1):134--162, 2008.

\bibitem[HP19]{Horing-Peternell-foliations}
Andreas H\"{o}ring and Thomas Peternell.
\newblock Algebraic integrability of foliations with numerically trivial
  canonical bundle.
\newblock {\em Invent. Math.}, 216(2):395--419, 2019.

\bibitem[HS05]{Huybrechts-Stellari}
Daniel Huybrechts and Paolo Stellari.
\newblock Equivalences of twisted {$K3$} surfaces.
\newblock {\em Math. Ann.}, 332(4):901--936, 2005.

\bibitem[Huy97]{Huybrechts2}
D.~Huybrechts.
\newblock {Birational symplectic manifolds and their deformations}.
\newblock {\em {J. Diff. Geom.}}, {45}:488--513, {1997}.

\bibitem[Kaw99]{Kawamata-Def}
Yujiro Kawamata.
\newblock Deformations of canonical singularities.
\newblock {\em J. Amer. Math. Soc.}, 12(1):85--92, 1999.

\bibitem[Kaw08]{Kawamata-flops}
Yujiro Kawamata.
\newblock Flops connect minimal models.
\newblock {\em Publ. Res. Inst. Math. Sci.}, 44(2):419--423, 2008.

\bibitem[KKV89]{Knop-Kraft-Vust}
Friedrich Knop, Hanspeter Kraft, and Thierry Vust.
\newblock The {P}icard group of a {$G$}-variety.
\newblock In {\em Algebraische {T}ransformationsgruppen und
  {I}nvariantentheorie}, volume~13 of {\em DMV Sem.}, pages 77--87.
  Birkh\"{a}user, Basel, 1989.

\bibitem[KL07]{Kaledin-Lehn}
D.~Kaledin and M.~Lehn.
\newblock Local structure of hyperk\"{a}hler singularities in {O}'{G}rady's
  examples.
\newblock {\em Mosc. Math. J.}, 7(4):653--672, 766--767, 2007.

\bibitem[KLS06]{Kaledin-Lehn-Sorger}
D.~Kaledin, M.~Lehn, and Ch. Sorger.
\newblock Singular symplectic moduli spaces.
\newblock {\em Invent. Math.}, 164(3):591--614, 2006.

\bibitem[KM92]{Kollar-Mori-flips}
J.~Koll\'{a}r and S.~Mori.
\newblock Classification of three-dimensional flips.
\newblock {\em J. Amer. Math. Soc.}, 5(3):533--703, 1992.

\bibitem[KM09]{Kuznetsov-Markushevich}
A.~Kuznetsov and D.~Markushevich.
\newblock Symplectic structures on moduli spaces of sheaves via the {A}tiyah
  class.
\newblock {\em J. Geom. Phys.}, 59(7):843--860, 2009.

\bibitem[Kuz10]{Kuznetsov-cubic}
A.~Kuznetsov.
\newblock Derived categories of cubic fourfolds.
\newblock In {\em Cohomological and geometric approaches to rationality
  problems}, volume 282 of {\em Progr. Math.}, pages 219--243. Birkh\"auser
  Boston, Inc., Boston, MA, 2010.

\bibitem[Kuz11]{Kuznetsov-base-change}
Alexander Kuznetsov.
\newblock Base change for semiorthogonal decompositions.
\newblock {\em Compos. Math.}, 147(3):852--876, 2011.

\bibitem[Kuz19]{Kuznetsov-Calabi-Yau}
Alexander Kuznetsov.
\newblock Calabi-{Y}au and fractional {C}alabi-{Y}au categories.
\newblock {\em J. Reine Angew. Math.}, 753:239--267, 2019.

\bibitem[Lie06]{Lieblich}
Max Lieblich.
\newblock Moduli of complexes on a proper morphism.
\newblock {\em J. Algebraic Geom.}, 15(1):175--206, 2006.

\bibitem[LLMS18]{Lahoz-Lehn-Macri-Stellari}
Mart\'{\i} Lahoz, Manfred Lehn, Emanuele Macr\`\i, and Paolo Stellari.
\newblock Generalized twisted cubics on a cubic fourfold as a moduli space of
  stable objects.
\newblock {\em J. Math. Pures Appl. (9)}, 114:85--117, 2018.

\bibitem[LLSvS17]{LLSvS}
Ch. Lehn, M.~Lehn, Ch. Sorger, and D.~van Straten.
\newblock Twisted cubics on cubic fourfolds.
\newblock {\em J. Reine Angew. Math.}, 731:87--128, 2017.

\bibitem[LP16]{Lehn-Pacienza}
Christian Lehn and Gianluca Pacienza.
\newblock Deformations of singular symplectic varieties and termination of the
  log minimal model program.
\newblock {\em Algebr. Geom.}, 3(4):392--406, 2016.

\bibitem[LPZ18]{Li-Pertusi-Zhao-twisted-cubics}
Chunyi Li, Laura Pertusi, and Xiaolei Zhao.
\newblock Twisted cubics on cubic fourfolds and stability conditions, 2018.

\bibitem[LPZ20]{Li-Pertusi-Zhao}
Chunyi Li, Laura Pertusi, and Xiaolei Zhao.
\newblock Elliptic quintics on cubic fourfolds, o'grady 10, and lagrangian
  fibrations, 2020.

\bibitem[LS06]{Lehn-Sorger}
M.~Lehn and C.~Sorger.
\newblock La singularit\'e de {O}'{G}rady.
\newblock {\em J. Algebraic Geom.}, 15(4):753--770, 2006.

\bibitem[LSV17]{LSV}
R.~Laza, G.~Sacc\`a, and C.~Voisin.
\newblock A hyper-k\"ahler compactification of the intermediate jacobian
  fibration associated with a cubic 4-fold.
\newblock {\em Acta Math.}, 218(1):55--135, 2017.

\bibitem[MS19]{Macri-Stellari-survey}
Emanuele Macr\`\i and Paolo Stellari.
\newblock Lectures on non-commutative {K}3 surfaces, {B}ridgeland stability,
  and moduli spaces.
\newblock In {\em Birational geometry of hypersurfaces}, volume~26 of {\em
  Lect. Notes Unione Mat. Ital.}, pages 199--265. Springer, Cham, [2019]
  \copyright 2019.

\bibitem[Muk84]{Mukai84}
S.~Mukai.
\newblock {Symplectic structure of the moduli space of sheaves on an abelian or
  $K3$ surface}.
\newblock {\em {Invent. Math.}}, {77}({1}):101--116, {1984}.

\bibitem[Muk87]{Mukai-Tata}
S.~Mukai.
\newblock On the moduli space of bundles on {$K3$} surfaces. {I}.
\newblock In {\em Vector bundles on algebraic varieties ({B}ombay, 1984)},
  volume~11 of {\em Tata Inst. Fund. Res. Stud. Math.}, pages 341--413. Tata
  Inst. Fund. Res., Bombay, 1987.

\bibitem[MZ16]{Meachan-Zhang}
Ciaran Meachan and Ziyu Zhang.
\newblock Birational geometry of singular moduli spaces of {O}'{G}rady type.
\newblock {\em Adv. Math.}, 296:210--267, 2016.

\bibitem[Nam01]{Namikawa-Note}
Yoshinori Namikawa.
\newblock A note on symplectic singularities, 2001.

\bibitem[Nam06]{Namikawa-Qfact}
Yoshinori Namikawa.
\newblock On deformations of {$\Bbb Q$}-factorial symplectic varieties.
\newblock {\em J. Reine Angew. Math.}, 599:97--110, 2006.

\bibitem[O'G97]{OGradyWt2}
Kieran~G. O'Grady.
\newblock The weight-two {H}odge structure of moduli spaces of sheaves on a
  {$K3$} surface.
\newblock {\em J. Algebraic Geom.}, 6(4):599--644, 1997.

\bibitem[{O'G}99]{OGrady99}
{O'Grady, K. G.}
\newblock Desingularized moduli spaces of sheaves on a {$K3$}.
\newblock {\em J. Reine Angew. Math.}, 512:49--117, 1999.

\bibitem[Per22]{Perry-Mukai}
Alexander Perry.
\newblock The integral {H}odge conjecture for two-dimensional {C}alabi-{Y}au
  categories.
\newblock {\em Compos. Math.}, 158(2):287--333, 2022.

\bibitem[PPZ19]{Perry-Pertusi-Zhao}
Alexander Perry, Laura Pertusi, and Xiaolei Zhao.
\newblock Stability conditions and moduli spaces for kuznetsov components of
  gushel-mukai varieties, 2019.

\bibitem[PR10]{Perego-Rapagnetta-OG10}
Arvid Perego and Antonio Rapagnetta.
\newblock Deformation of the o'grady moduli spaces, 2010.

\bibitem[PR18]{Perego-Rapagnetta-Irr}
Arvid Perego and Antonio Rapagnetta.
\newblock The moduli spaces of sheaves on k3 surfaces are irreducible
  symplectic varieties, 2018.

\bibitem[PR20]{Perego-Rapagnetta-second-cohomology}
Arvid Perego and Antonio Rapagnetta.
\newblock The second integral cohomology of moduli spaces of sheaves on k3 and
  abelian surfaces, 2020.

\bibitem[{Sac}23]{IntJac}
G.~{Sacc\`a}.
\newblock Birational geometry of the {I}ntermediate {J}acobian fibration of a
  cubic fourfold.
\newblock {\em to appear Geometry and Topology}, 2023.

\bibitem[Sch20]{Schwald-Definition}
Martin Schwald.
\newblock On the definition of irreducible holomorphic symplectic manifolds and
  their singular analogs, 2020.

\bibitem[{Sta}23]{StacksProject}
The {Stacks Project Authors}.
\newblock Stacks project, 2023.
\newblock \url{https://stacks.math.columbia.edu}.

\bibitem[Tod08]{Toda}
Y.~Toda.
\newblock Moduli stacks and invariants of semistable objects on {$K3$}
  surfaces.
\newblock {\em Adv. Math.}, 217(6):2736--2781, 2008.

\bibitem[uTT83]{Le-Teissier}
L\^{e}~D\ {u}ng Tr\'{a}ng and B.~Teissier.
\newblock Cycles evanescents, sections planes et conditions de {W}hitney. {II}.
\newblock In {\em Singularities, {P}art 2 ({A}rcata, {C}alif., 1981)},
  volume~40 of {\em Proc. Sympos. Pure Math.}, pages 65--103. Amer. Math. Soc.,
  Providence, RI, 1983.

\bibitem[Ver76]{Verdier}
Jean-Louis Verdier.
\newblock Stratifications de {W}hitney et th\'{e}or\`eme de {B}ertini-{S}ard.
\newblock {\em Invent. Math.}, 36:295--312, 1976.

\bibitem[Voi07]{Voisin-Some-Aspects}
C.~Voisin.
\newblock Some aspects of the {H}odge conjecture.
\newblock {\em Jpn. J. Math.}, 2(2):261--296, 2007.

\bibitem[Voi13]{Voisin-Handbook}
Claire Voisin.
\newblock Hodge loci.
\newblock In {\em Handbook of moduli. {V}ol. {III}}, volume~26 of {\em Adv.
  Lect. Math. (ALM)}, pages 507--546. Int. Press, Somerville, MA, 2013.

\bibitem[Voi18]{Voisin-Twisted}
C.~Voisin.
\newblock Hyper-{K}\"{a}hler compactification of the intermediate {J}acobian
  fibration of a cubic fourfold: the twisted case.
\newblock In {\em Local and global methods in algebraic geometry}, volume 712
  of {\em Contemp. Math.}, pages 341--355. Amer. Math. Soc., Providence, RI,
  2018.

\bibitem[VP21]{Villalobos}
David Villalobos-Paz.
\newblock Moishezon spaces and projectivity criteria, 2021.

\bibitem[Yos01]{Yoshioka-moduli-abelian}
Kota Yoshioka.
\newblock Moduli spaces of stable sheaves on abelian surfaces.
\newblock {\em Math. Ann.}, 321(4):817--884, 2001.

\end{thebibliography}

\end{document}